\documentclass[reqno]{amsart}

\usepackage{amsthm,amsmath,amsfonts,amssymb,epsfig,graphicx}
\usepackage{dsfont,bbm}
\usepackage{enumitem}

\newtheorem{definition}{Definition}
\newtheorem{theorem}[definition]{Theorem}
\newtheorem{proposition}[definition]{Proposition}
\newtheorem{lemma}[definition]{Lemma}

\newtheorem{assumption}[definition]{Assumption}
\newtheorem{remark}[definition]{Remark}

\newtheorem{notation}[definition]{Notation}

\newcommand{\dl}{\ensuremath{\mathrm{d}}}

\begin{document}

\title[Stochastic fiber dynamics in a spatially semi-discrete setting]{Stochastic dynamics for inextensible fibers \\ in a spatially semi-discrete setting}

\author[F. Lindner]{Felix Lindner$^1$}
\author[N. Marheineke]{Nicole Marheineke$^2$}
\author[H. Stroot]{Holger Stroot$^{1,3}$}
\author[A. Vibe]{Alexander Vibe$^2$}
\author[R. Wegener]{Raimund Wegener$^3$}

\date{\today\\
$^1$ TU Kaiserslautern, Fachbereich Mathematik, D-67653 Kaiserslautern, Germany\\
$^2$ FAU Erlangen-N\"urnberg, Lehrstuhl Angewandte Mathematik I, Cauerstr.\ 11, D-91058 Erlangen, Germany\\
$^3$ Fraunhofer ITWM, Fraunhofer Platz 1, D-67663 Kaiserslautern, Germany}

\begin{abstract}
We investigate a spatially discrete surrogate model for the dynamics of a slender, elastic, inextensible fiber in turbulent flows. Deduced from a continuous space-time beam model for which no solution theory is available, it consists of a high-dimensional second order stochastic differential equation in time with a nonlinear algebraic constraint and an associated Lagrange multiplier term. We establish a suitable framework for the rigorous formulation and analysis of the semi-discrete model and prove existence and uniqueness of a global strong solution. The proof is based on an explicit representation of the Lagrange multiplier and on the observation that the obtained explicit drift term in the equation satisfies a one-sided linear growth condition on the constraint manifold. The theoretical analysis is complemented by numerical studies concerning the time discretization of our model. The performance of implicit Euler-type methods can be improved when using the explicit representation of the Lagrange multiplier to compute refined initial estimates for the Newton method applied in each time step.
\end{abstract}

\maketitle

\quad\\
\textsc{AMS-Classification.} Primary 60H10, 74K10, 74Hxx; Secondary 58J65, 65C30. 

%%%%%%%%%%%%%%%%%%%%%%%%%%%%%%%%
%%%
\section{Introduction}

The stochastic dynamics of thin long elastic fibers is of interest in various applications ranging from biomolecular science to paper and technical textile manufacturing \cite{mesirov:b:1996, pearson:b:1985, wegener:b:2015}. In the slender body theory a fiber can be asymptotically described by an arc-length parameterized, time-dependent random curve $\mathbf{r}$ representing its centerline. Its dynamics due to acting deterministic and stochastic forces, such as gravity, friction, turbulent aero- or hydrodynamics, can be modeled by a system of nonlinear partial differential equations driven by a multiplicative vector-valued space-time white noise $\boldsymbol{\xi}$, \cite{marheineke:p:2011}
\begin{equation}\label{eq:intro}
\begin{aligned}
\partial_t((\rho A) \partial_{t} \mathbf{r})&= \partial_s (\lambda \partial_s \mathbf{r}) - \partial_{ss}((EI) \partial_{ss} \mathbf{r}) + \mathbf{f}(\mathbf{r}, \partial_t \mathbf{r}, \partial_s \mathbf{r}, t)+ \mathbf{A}(\mathbf{r}, \partial_t \mathbf{r},\partial_s \mathbf{r}, t)  \,\boldsymbol{\xi}, \\   \|\partial_s \mathbf{r}(s, t)\| &= 1.
\end{aligned}
\end{equation}
The arc-length constraint enforces local inextensibility and hence the global conservation of length. It turns the scalar-valued inner traction $\lambda$ to an unknown random parameter, i.e., Lagrange multiplier. The system for $(\mathbf{r}, \lambda)$ has a beam-type character due to inertia (line weight $(\rho A)$) with an elliptic regularization coming from the bending stiffness $(EI)$. Its deterministic version ($\mathbf{A}\equiv\mathbf{0}$) can be viewed as a reformulation of the Kirchhoff-Love equations \cite{landau:b:1970} that describe the asymptotic limit model of an elastic Euler-Bernoulli rod as the slenderness parameter (ratio between fiber diameter and length) and the Mach number (ratio between fiber velocity and typical speed of sound) approach zero \cite{baus:p:2015}. For rigorous derivations of such inextensible Kirchhoff beam models from three-dimensional hyper-elasticity we refer to, for example, \cite{coleman:p:1993, mora:p:2003}. 

Fast and accurate numerical simulations are strongly required in nonwoven manufacturing, for the exploration, design and optimization of turbulent fiber lay-down processes and their resulting fabric quality, \cite{bonilla:p:2007, klar:p:2009, marheineke:p:2011}. So far, the used approaches were mainly addressed to high-speed performance without any theoretical results on convergence or length conservation. As far as we know, there exist no analytical results for the constrained stochastic system \eqref{eq:intro}, neither concerning the convergence of numerical schemes nor concerning the existence and uniqueness of a solution. The solvability of extensible stochastic beam equations without constraint is studied in, among others, \cite{baur:p:2013, brzezniak:p:2005, daprato:b:2014}, see also the references therein. The solvability of the deterministic counterpart of the constrained system \eqref{eq:intro} (and variants thereof) is investigated in \cite{dziuk:p:2002, grothaus:p:2015, oelz:p:2011, preston:p:2011}. Deterministic elastic flows of constrained curves in different model variants are also a topic of recent numerical investigations, see \cite{barrett:p:2011, bartels:p:2013, deckelnick:p:2009, grothaus:p:2015}. Considering a global length constraint, an error analysis for a spatially semi-discrete scheme was performed in \cite{deckelnick:p:2009}, a fully implicit finite element method with equidistribution properties was explored in \cite{barrett:p:2011}. As for the deterministic version to \eqref{eq:intro}, the nonlinear pointwise constraint of the local length preservation was handled by a linearization around a previous solution in each time step which led to a sequence of linear saddle point problems in space in \cite{bartels:pp:2014, grothaus:p:2015}. A rigorous convergence analysis for the temporally semi-discrete setting can be found in \cite{grothaus:p:2015}.

The focus of this paper is a better understanding for the constrained stochastic partial differential equation \eqref{eq:intro} and its numerical behavior by a thorough analytical investigation of the corresponding spatially semi-discretized formulation, complemented by numerical studies concerning its time discretization. In particular, we provide a suitable framework for the rigorous interpretation of the spatially discrete model. It is formally derived by applying a finite volume approach with constant cell size $\Delta s$ and describes the discrete counterpart to the fiber curve $\mathbf r(\cdot,t)$ via a finite sequence of associated fiber points $r(t)=(\mathbf r_i(t))_{i=1,\ldots,N}$ in terms of a polygon line for each time $t\geq 0$. The model can be interpreted as an It\^o-type high-dimensional stochastic differential equation (SDE) of the form
\begin{equation}\label{eq:intro_SDE}
\begin{aligned}
 \dl r &= v\,\dl t,\\
 \dl v &= a(t,r,v)\,\dl t + B(t,r,v)\,\dl w(t) + \nabla g(r) \,\dl \mu(t),\\
 g(r(t)) &= 0, 
 \end{aligned}
\end{equation}
with velocity $v$, driving vector-valued Brownian motion $w$, involved drift and diffusion functions $a$ and $B$ as well as a nonlinear constraint function $g$. The process $\mu$ is a vector-valued continuous semimartingale which serves as a Lagrange multiplier to the inextensibility constraint $g(r(t))=0$, so that, formally, the time derivative of $\mu(t)$ corresponds to $\lambda(\cdot,t)$ in \eqref{eq:intro}, cf.\ Sec.~\ref{sec:model} for details. The semi-discrete system \eqref{eq:intro_SDE} should be considered as a simplified surrogate model for the full dynamical system \eqref{eq:intro} which, though simplified, inherits characteristic features and difficulties of the full model, in particular the presence of a nonlinear algebraic constraint and a corresponding Lagrange multiplier term, and at the same time allows for a rigorous analysis in terms of It\^o calculus. Note that the constraint gives rise to a submanifold $\mathcal M$ of the state space, so that \eqref{eq:intro_SDE} is in fact a manifold-valued SDE. 

Constrained or manifold-valued SDEs similar to \eqref{eq:intro_SDE} have been considered in the context of molecular dynamics, see, for example, \cite{ciccotti:p:2006,kallemov:p:2011,lelievre:p:2012} and the references therein. It is known that the Lagrange multiplier $\mu$ can be represented explicitly in terms of $r,v,a,B,w$ and $g$, yielding the existence of a local solution, cf.\ the solution theory in \cite{hsu:b:2002}. However, to the best of our knowledge, the question whether there exists a global (strong) solution to SDEs of the type \eqref{eq:intro_SDE}, i.e., whether the local solution has infinite lifetime, has not been treated yet. This question is non-trivial since, in its explicit formulation, the drift term in the equation shows a quite involved quadratic dependence on the velocity $v$, cf.\ Sec.~\ref{sec:existence_and_uniqueness}. In this paper, we provide a rigorous proof of the existence and uniqueness of a global strong solution to \eqref{eq:intro_SDE}. 
%It is based on the observation that the drift coefficient of the explicitly reformulated equation enjoys a one-sided linear growth property on the constraint manifold $\mathcal M$. 
It is based on a detailed analysis of the drift coefficient of the explicitly reformulated equation, which turns out to satisfy a one-sided linear growth property on the constraint manifold $\mathcal M$.
In combination with a  Gronwall-type argument, this is used to verify that the lifetime of the local solution is infinite with probability one.

Insights from the theoretical analysis allow the improvement of the numerical treatment of \eqref{eq:intro}. 
We remark that rigorous strong convergence results for constraint-preserving time discretizations of SDEs of the type \eqref{eq:intro_SDE} are not available in the literature, but various numerical schemes have been proposed, see, for example, \cite{ciccotti:p:2006,kallemov:p:2011,lelievre:p:2012} and compare also \cite{ciccotti:p:2008,lelievre:p:2008}. 
Recently, strong convergence results for numerical schemes for a class of explicitly given SDEs with a conserved quantity have been derived in \cite{hong:p:2011,zhou:p:2016} (compare also \cite{milstein:p:2002}), but the assumptions on the coefficients therein are not satisfied by the coefficients appearing in the explicit reformulation of \eqref{eq:intro_SDE}.
In this paper, we investigate numerically the strong convergence behavior of an implicit Euler scheme and an explicit projection-based variant and discuss their applicability to the turbulence-driven fiber dynamics regarding accuracy and computational effort. We show that the performance of the implicit Euler-type methods can be speeded up when using the explicit representation of the Lagrange multiplier to compute refined initial estimates for the Newton method in each time step --in the spirit of a predictor-corrector scheme, cf.\ WIGGLE-algorithm in molecular dynamics \cite{lee:p:2005}. In that context we refer to \cite{griebel:b:2007} and references therein for symplectic schemes for constrained (stochastic) Hamiltonians, e.g., RATTLE and SHAKE-algorithms.

The paper is structured as follows. Starting with the continuous space-time model for the stochastic fiber dynamics and its spatial semi-discretization, we embed the resulting spatially discrete model into the framework of manifold-valued It\^o-type SDEs in Sec.~\ref{sec:model}.  We present the global existence and uniqueness result for the corresponding class of SDEs with holonomic constraint in Sec.~\ref{sec:existence_and_uniqueness} and discuss the implications for the numerical handling of our fiber model in Sec.~\ref{sec:numerics}.

%%%%%%%%%%%%%%%%%%%%%%%%
%%%
\section{Model for the Stochastic Fiber Dynamics}
\label{sec:model}
\setcounter{equation}{0}

In this section we present a stochastic inextensible beam model for the dynamics of a slender elastic inertial fiber immersed in a turbulent flow field. Following the works \cite{marheineke:p:2006, marheineke:p:2011}, the turbulent effects are particularly incorporated by means of a stochastic force model in terms of multiplicative space-time white noise. The applied spatial semi-discretization leads to a high-dimensional system of It\^o-type SDEs with holonomic constraint.

%%%
\subsection{Turbulence-driven fiber -- continuous space-time model}
The characteristic feature of a fiber is its long slender geometry. Asymptotically, an elastic fiber can be modeled as an inextensible Kirchhoff beam, because extension and shear are negligibly small compared to bending. Torsion plays also no role in case that one fiber end is free. In the slenderbody theory the fiber is represented by an arc-length parameterized time-dependent curve $\mathbf{r}:[0,\ell]\times\mathbb{R}_0^+\to \mathbb{R}^d,\;(s,t)\mapsto \mathbf r(s,t)$, e.g., its centerline with fiber length $\ell$. If randomness is involved, it becomes a random field $\mathbf{r}:\Omega\times[0,\ell]\times\mathbb{R}_0^+\to \mathbb{R}^d,\;(\omega,s,t)\mapsto \mathbf r(\omega,s,t)$, where $d$ denotes the space dimension, $d\in\{2,3\}$. We consider an homogeneous inertial fiber in a turbulent flow field under gravity and describe its dynamics by help of a constrained stochastic partial differential system with multiplicative space-time white noise, $\boldsymbol{\xi}(s,t)\in \mathbb{R}^d$, in the spirit of \cite{marheineke:p:2011}. In dimensionless form the stochastic fiber model for position $\mathbf{r}$, velocity $\mathbf{v}$ and inner traction $\lambda$  is given by
\begin{align}\label{eq:intro_fullsystem} \nonumber
\partial_t \mathbf{r}&=\mathbf{v},\\\nonumber
\partial_t \mathbf{v}&=  \partial_s(\lambda \partial_s\mathbf{r})-\alpha^{-2} \partial_{ssss}\mathbf{r}+\mathrm{Fr}^{-2} \mathbf{e}_{\mathrm g}+\mathrm{Dr}^{-2}\mathbf{C}(\mathbf{r},\partial_s\mathbf{r},t)\,[(\mathbf{u}(\mathbf{r},t)-\mathbf{v})+\beta \mathbf{D}(\mathbf{r},\mathbf{v},\partial_s\mathbf{r},t) \,\boldsymbol{\xi}],\\
\|\partial_s\mathbf{r}\|&=1,
\end{align}
supplemented with appropriate initial and boundary conditions. We particularly use
\begin{subequations}\label{eq:intro_bc}
\begin{align}\label{eq:intro_bc_fs}
 \mathbf{r}(0,t)=\mathbf{\hat r},\quad \partial_s\mathbf{r}(0,t)=\boldsymbol{\hat \tau}, \,\, \|\boldsymbol{\hat \tau}\|=1, \qquad \lambda(\ell,t)=0,\quad \partial_{ss}\mathbf{r}(\ell,t)=\mathbf{0}, \quad \partial_{sss}\mathbf{r}(\ell,t)=\mathbf{0}
\end{align}
in case of a fiber clamped at one end ($s=0$) and stress-free at the other end ($s=\ell$) as well as
\begin{align}\label{eq:intro_bc_ff}
 \lambda(s,t)=0, \quad \partial_{ss}\mathbf{r}(s,t)=\mathbf{0}, \quad \partial_{sss}\mathbf{r}(s,t)=\mathbf{0}, \quad s\in\{0,\ell\}
\end{align}
\end{subequations}
in case of a free moving fiber. Here, the inner traction $\lambda$ is a  real-valued (generalized) random field on $[0,\ell]\times\mathbb R_0^+$ that acts as a Lagrange multiplier to the nonlinear pointwise constraint. The latter is expressed in the Euclidean norm $\|\cdot\|$ and ensures the arc-length parameterization for all times. It enforces the local inextensibility and hence the global conservation of length. The fiber dynamics is driven by the inner forces due to traction and bending and by the acting outer forces due to gravity and drag. The drag forces arising in a turbulent flow are modeled as superposition of a deterministic and a stochastic part, \cite{marheineke:p:2006, marheineke:p:2011}. Whereas the deterministic force is determined by the relative velocity between the mean flow velocity $\mathbf{u}$ (evaluated at $(\mathbf r(s,t),t)$) and the fiber velocity, the stochastic force takes into account the impact of the turbulent fluctuations. It is modeled as three-dimensional space-time white noise $\boldsymbol{\xi}$ with flow-dependent amplitude $\mathbf{D}$. We use a linear (Stokes-type) drag model where the $\mathbb{R}^{d\times d}$-valued drag operators $\mathbf{C}$ and $\mathbf{D}$ depend explicitly on the fiber velocity and/or orientation; the further dependence on the flow field enters via the evaluation of the flow quantities at $(\mathbf{r},t)$, cf.,\ $\mathbf{u}(\mathbf{r},t)$. The stochastic fiber system \eqref{eq:intro_fullsystem} is characterized by four dimensionless numbers: the bending number $\alpha$ (ratio of inertial and bending forces), the Froude number $\mathrm{Fr}$ (ratio of inertial and gravitational forces with gravitational direction $\mathbf{e}_{\mathrm g}$, $\|\mathbf{e}_{\mathrm g}\|=1$), the drag number $\mathrm{Dr}$ (ratio of inertial and mean drag forces) and the turbulent fluctuation number $\beta$. 

\begin{assumption}[Drag force model]\label{ass:1} For the subsequent investigations we assume that the flow velocity $\mathbf u:\mathbb{R}^d\times \mathbb{R}_0^+ \to\mathbb{R}^d$ is continuously differentiable with at most linear growth in the space argument. Moreover, the drag operators are continuously differentiable and bounded, i.e., $\mathbf{C}:\mathbb{R}^{2d} \times \mathbb{R}_0^+ \to \mathbb{R}^{d\times d}$ and $\mathbf{D}:\mathbb R^{3d}\times \mathbb{R}_0^+\to \mathbb{R}^{d\times d}$. Note that the consideration of nonlinear drag models as, e.g., given in \cite{marheineke:p:2011} is possible, presupposing respective regularity and growth conditions.
\end{assumption}

%%%
\subsection{Spatial semi-discretization}
\label{sec:semi_disc}

The fiber model \eqref{eq:intro_fullsystem} is a wave-like system with elliptic regularization. For the spatial discretization of the deterministic version different approaches can be found in literature, such as, e.g., geometric Lagrangian methods \cite{langer:p:1996}, finite element schemes \cite{barrett:p:2011, bartels:p:2013, grothaus:p:2015} or finite volume approaches \cite{wegener:b:2015}. We apply a finite volume method in combination with a finite difference approximation for the constraint. The usage of a staggered grid allows particularly for small discretization stencils. In the deterministic case, it is a conservative first-order scheme. However, note that the derivation of the spatially semi-discrete stochastic system is here purely formal, a rigorous interpretation follows in Sec.~\ref{sec:interpretation_as_SDE}.

We consider a finite volume discretization that is based on a conforming partition of $[0,\ell]$ into subintervals (control cells) $I_i$ of length $\Delta s$.  The partition is identified with the sequence of nodes $\{s_i\}$. The idea is to formally integrate the evolution equations in \eqref{eq:intro_fullsystem} over the control cells $I_i=[s_{i-1/2},s_{i+1/2}]$, $s_{i\pm 1/2}=s_i\pm \Delta s/2$ for $i=1,...,N$, and to set up a stochastic differential system in time for the cell averages $\boldsymbol{\varphi}_i(t)=\int_{I_i} \boldsymbol{\varphi}(s,t) \dl s/\Delta s$ of the unknowns $\boldsymbol{\varphi}\in\{\mathbf{r}, \mathbf{v}\}$. In this way the fiber position and velocity are assigned to the cell nodes, $\mathbf{r}_i(t)$, $\mathbf{v}_i(t)$, whereas we assign the inner traction --and consequently also the constraint-- to the cell edges, $\lambda_{i\pm1/2}(t)=\lambda(s_{i\pm1/2},t)$. Proceeding from the integral equations
\begin{equation}\label{eq:semidiscretization3} 
\begin{aligned}
\frac{\dl}{\dl t} \mathbf{r}_i(t) &= \mathbf{v}_i(t), \\
\frac{\dl}{\dl t}\mathbf{v}_i(t) &= \frac{1}{\Delta s} \left( \boldsymbol{\phi}(s_{i+1/2},t) - \boldsymbol{\phi}(s_{i-1/2},t) + \int_{I_i}\mathbf{f}(s,t) \,\dl s +  \int_{I_i} \mathbf{A}(s,t)  \boldsymbol{\xi}(s,t)\,\dl s \right),\\ 
\|\partial_s\mathbf{r}(s_{i-1/2},t)\|&=1,
\end{aligned}
\end{equation}
where we abbreviate the flux function by $\boldsymbol{\phi}=\lambda \partial_s \mathbf{r} - \alpha^{-2}\partial_{sss} \mathbf{r}$, the deterministic source terms by $\mathbf{f}=\mathrm{Fr}^{-2} \mathbf{e}_{\mathrm g}+\mathrm{Dr}^{-2}\mathbf{C}(\mathbf{r},\partial_s\mathbf{r},.)(\mathbf{u}(\mathbf{r},.)-\mathbf{v})$ and the noise amplitude by $\mathbf{A}=\mathrm{Dr}^{-2}\beta \mathbf{C}(\mathbf{r},\partial_s\mathbf{r},.)\mathbf{D}(\mathbf{r},\mathbf{v},\partial_s\mathbf{r},.)$, we evaluate the integrals by means of trapezoidal quadrature rules,
\begin{align*}
\int_{I_i}\mathbf{f}(s,t) \,\dl s  &\approx \frac{1}{2}(\mathbf{f}(s_{i+1/2},t)+\mathbf{f}(s_{i-1/2},t))\Delta s, \\
 \int_{I_i} \mathbf{A}(s,t)\boldsymbol{\xi}(s,t) \, \dl s &\approx \frac{1}{2}(\mathbf{A}(s_{i+1/2},t) +\mathbf{A}(s_{i-1/2},t))   \int_{I_i} \boldsymbol{\xi}(s,t) \, \dl s.
\end{align*}
Moreover, we approximate all function values of $\mathbf r$, $\mathbf v$ as well as $\partial_s \mathbf{r}$, $\partial_{ss}\mathbf{r}$, $\partial_{sss} \mathbf{r}$ occurring at the cell edges $s_{i\pm1/2}$ by a linear interpolation over the neighboring cells and first order finite differences, respectively, i.e.,
\begin{align*}
\boldsymbol{\varphi}(s_{i-1/2},t)&\approx(\boldsymbol{\varphi}_{i}(t)+\boldsymbol{\varphi}_{i-1}(t))/2,\qquad \boldsymbol{\varphi}\in\{\mathbf{r},\mathbf{v}\},\\
\partial_s \mathbf{r}(s_{i-1/2},t)&\approx (\mathbf{r}_{i}(t)-\mathbf{r}_{i-1}(t))(\Delta s)^{-1},\\
\partial_{ss} \mathbf{r}(s_{i-1/2},t)&\approx (\mathbf{r}_{i+1}(t)-\mathbf{r}_{i}(t)-\mathbf{r}_{i-1}(t)+\mathbf{r}_{i-2}(t))(\Delta s)^{-2},\\
\partial_{sss} \mathbf{r}(s_{i-1/2},t)&\approx (\mathbf{r}_{i+1}(t)-3\mathbf{r}_{i}(t)+3\mathbf{r}_{i-1}(t)-\mathbf{r}_{i-2}(t))(\Delta s)^{-3}.
\end{align*}
Consequently, plugging the discretization stencils into the integral equations yields a semi-discrete system for $\mathbf{r}_i$, $\mathbf{v}_i$ and $\lambda_{i-1/2}$. The traction $\lambda_{i-1/2}(t)$ appearing in the flux term acts particularly as Lagrange multiplier to the discretized inextensibility constraint 
\begin{align*}
\|\mathbf{r}_{i}(t)-\mathbf{r}_{i-1}(t)\|=\Delta s.
\end{align*}
Due to the constraint, the spatially discrete fiber becomes a polygon line with a fixed geometrical spacing for the fiber points. In the sequel we denote the used approximations of the function values at the cell edges by the index $_{i\pm 1/2}$, e.g., $\mathbf{A}_{i\pm1/2}(t)\approx \mathbf{A}(s_{i\pm 1/2},t)$. 
 
To incorporate the boundary conditions \eqref{eq:intro_bc} into the scheme, we introduce ghost points, i.e., artificial points which are not governed by the dynamical system~\eqref{eq:intro_fullsystem}. This way we do not have to adapt the discretization stencils. The spatial grid for a one-sided clamped fiber \eqref{eq:intro_bc_fs} is visualized in Fig.~\ref{fig:app_disc}. We realize the clamped boundary condition of a fixed position $\mathbf{\hat{r}}$ and tangent $\boldsymbol{\hat\tau}$ at $s=0$ by setting the node $s_1=1.5\Delta s$, where $\Delta s=\ell/(N+1)$, and introducing the additional nodes $s_{-1}=-0.5\Delta s$ and $s_0 = 0.5\Delta s$. The associated fiber points $\mathbf{r}_{i}$, $i=-1,0$ are then given algebraically at every time $t\geq0$, they are coupled with the dynamical system via the constraint for the unknown traction $\lambda_{1/2}$,
\begin{align*}
 \mathbf{r}_{-1}+\mathbf{r}_0 &= 2\mathbf{\hat{r}},\qquad
 \mathbf{r}_0-\mathbf{r}_{-1} = \Delta s\boldsymbol{\hat{\tau}}, \qquad \|\mathbf{r}_1-\mathbf{r}_0\|=\Delta s.
\end{align*}
In particular, $\mathbf r_0=\mathbf{\hat{r}}+\Delta s \boldsymbol{\hat{\tau}}/2$ is constant and hence considered as a ghost point although $s_0\in[0,\ell]$. For the realization of a stress-free boundary at $s=\ell$, we set $s_N=\ell-0.5\Delta s$ and add $s_{N+i} = \ell+(i-0.5)\Delta s$, $i=1,2$. According to the discretization stencils the corresponding fiber points and traction fulfill
\begin{align*}
 \mathbf{r}_{N+2}-3\mathbf{r}_{N+1}+3\mathbf{r}_{N}-\mathbf{r}_{N-1} &= \mathbf{0},\qquad
 \mathbf{r}_{N+2}-\mathbf{r}_{N+1}-\mathbf{r}_{N}+\mathbf{r}_{N-1} = \mathbf{0}, \qquad \lambda_{N+1/2}=0.
\end{align*}
Note that this approach also ensures that the additional points fulfill the length constraint. In case of a free moving fiber \eqref{eq:intro_bc_ff} the handling of the stress-free boundary conditions for both ends is straightforward.

In the paper we will focus on the analysis of the one-sided clamped fiber as this case covers the difficulties of both types of boundary conditions (clamped and stress-free). 

\begin{figure}[h!]
 \includegraphics[scale=0.9]{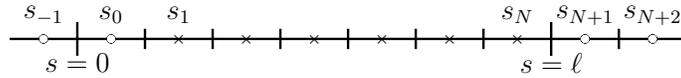}
 \caption{\small Spatial grid for semi-discretization of one-sided clamped fiber. Crosses mark the dynamical grid points, whereas circles label the ghost points due to the realization of the boundary conditions \eqref{eq:intro_bc_fs}.}
\label{fig:app_disc}
\end{figure}

%%%
\subsection{Interpretation as It\^o-type SDE}
\label{sec:interpretation_as_SDE}

So far the deduced semi-discrete system is purely formal. In the momentum balance, neither the evaluation of the space-time white noise $\boldsymbol\xi$ at fixed points $(s,t)\in[0,\ell]\times\mathbb R_0^+$ is defined nor the rigorous meaning of the Lagrange multiplier is clear. We show that the model can be interpreted as a constrained system of It\^o-type stochastic differential equations which will serve us as a simplified surrogate for the continuous space-time model \eqref{eq:intro_fullsystem}.  To establish a suitable framework for its analytical investigation in Sec.~\ref{sec:existence_and_uniqueness} we recall some basics from stochastic analysis.

Throughout the paper, all random objects are supposed to be defined on a common, sufficiently rich, complete probability space $(\Omega,\mathcal F,\mathbb P)$. By $\mathcal B_0([0,\ell]\times\mathbb R_0^+)$ we denote the system of all Borel sets $A\subset\mathcal B([0,\ell]\times\mathbb R_0^+)$ with finite Lebesgue mass $|A|=\int_{[0,\ell]\times\mathbb{R}^+_0}\mathbbm{1}_A(s,t)\,\dl(s,t)$. 

\begin{definition}[Gaussian white noise \cite{walsh:b:1986}] 
A $\mathbb R^d$-valued \emph{Gaussian white noise} on $[0,\ell]\times\mathbb R_0^+$, with Lebesgue measure as reference measure, is a mapping $\boldsymbol\xi:\mathcal B_0([0,\ell]\times\mathbb R_0^+)\to L^2(\Omega,\mathcal F,\mathbb P;\mathbb R^d)$ such that
\begin{itemize}
\item[(i)] each $\boldsymbol\xi(A)$ is Gaussian distributed with mean $\mathbf 0\in\mathbb R^d$ and covariance matrix $|A|\,\mathbf I_d\in\mathbb R^{d\times d}$ ($\mathbf I_d$ is the identity matrix);
\item[(ii)] if $A_1\cap\ldots\cap A_n=\emptyset$, then $\boldsymbol\xi(A_1),\ldots,\boldsymbol\xi(A_n)$ are independent and $\boldsymbol\xi\big(\bigcup_{i=1}^n A_i\big)=\sum_{i=1}^n\boldsymbol\xi(A_i)$.
\end{itemize}
\end{definition}

Understanding $\boldsymbol\xi$ in \eqref{eq:intro_fullsystem} and \eqref{eq:semidiscretization3} in this way, we do not deal with a regular random field on $[0,\ell]\times\mathbb R_0^+$ but with a generalized random field or a random measure. An It\^o-type stochastic integral $\int_{[0,\ell]\times\mathbb R^+_0} \mathbf F(s,t)\,\boldsymbol\xi(\dl s,\dl t)$ with respect to this random measure can be defined for predictable $\mathbb R^{d\times d}$-valued integrands $\mathbf F$ satisfying a suitable integrability assumption, see, e.g.,~\cite{walsh:b:1986} for details. Here and below, predictability and adaptedness refer to the underlying filtration $(\mathcal F_t)_{t\geq0}$ given by
\begin{equation*}
\mathcal F_t:=\sigma\big(\boldsymbol\xi(B\times[0,t]):\,B\in\mathcal B([0,\ell])\big),\quad t\geq0.
\end{equation*} 
Thus, when formally integrating the semi-discrete momentum balance with respect to the time variable $t$ over a finite interval $[0,\tau]$, we can interpret the expression
\begin{subequations} \label{eq:interpretation_as_SDE_1}
\begin{align} \label{eq:interpretation_as_SDE_1a}
\int_0^\tau\Big(\frac{1}{2\Delta s}\big(\mathbf{A}_{i+1/2}(t) +\mathbf{A}_{i-1/2}(t)\big)  \int_{I_i} \boldsymbol{\xi}(s,t) \, \dl s\Big)\dl t
\end{align} 
as a stochastic integral $\int_{I_i\times[0,\tau]}(2\Delta s)^{-1}(\mathbf{A}_{i+1/2}(t) +\mathbf{A}_{i-1/2}(t))\,   \boldsymbol{\xi}(\dl s,\dl t)$ with respect to $\boldsymbol\xi$, given that $\mathbf{A}_{i\pm1/2}$ are, for example, adapted and continuous. Since $\mathbf{A}_{i\pm1/2}$ do not depend on $s$, this simplifies to an It\^o-type stochastic integral $\int_0^\tau(2\Delta s)^{-1}(\mathbf{A}_{i+1/2}(t) +\mathbf{A}_{i-1/2}(t))\,   \boldsymbol{\xi}(I_i,\dl t) $ w.r.t.\ the random measure $\mathcal B_0(\mathbb R^+_0)\ni B\mapsto \boldsymbol\xi(I_i\times B)\in L^2(\Omega,\mathcal F,\mathbb P;\mathbb R^d)$ on $\mathbb R^+_0$. Introducing $\mathbb R^d$-valued standard Wiener processes $\mathbf w_i=(\mathbf w_i(t))_{t\geq0}$, $i=1,\ldots,N$, by setting
\begin{align}\label{eq:w_i}
\mathbf w_i(t):=\frac{1}{\sqrt{\Delta s}}\,\boldsymbol\xi(I_i\times[0,t]),\quad t\geq0,
\end{align}
we have the identity
\begin{align}\label{eq:interpretation_as_SDE_1b}
\int_0^\tau\frac{1}{2\Delta s}\big(\mathbf{A}_{i+1/2}(t) +\mathbf{A}_{i-1/2}(t)\big)\,   \boldsymbol{\xi}(I_i,\dl t)
= \int_0^\tau\frac{1}{2\sqrt{\Delta s}}\big(\mathbf{A}_{i+1/2}(t) +\mathbf{A}_{i-1/2}(t)\big)\,\dl\mathbf w_i(t)
\end{align}
\end{subequations}
with a standard It\^o integral on the right hand side as our rigorous interpretation of \eqref{eq:interpretation_as_SDE_1a}.

The time integration of the semi-discrete momentum balance leads to a further term which cannot be interpreted in a standard way. Namely, the flux term  $(\Delta s)^{-1}(\boldsymbol{\phi}_{i+1/2}-\boldsymbol{\phi}_{i-1/2})$ involves
\begin{subequations}\label{eq:interpretation_as_SDE_2}
\begin{align}\label{eq:interpretation_as_SDE_2a}
\frac1{\Delta s}\int_0^\tau \Big(\lambda_{i+1/2}(t)\frac{\mathbf r_{i+1}(t)-\mathbf r_i(t)}{\Delta s}-\lambda_{i-1/2}(t)\frac{\mathbf r_{i}(t)-\mathbf r_{i-1}(t)}{\Delta s}\Big)\,\dl t,
\end{align}
which cannot be defined as a Lebesgue integral either since, loosely speaking, the Lagrange multipliers $\lambda_{i\pm 1/2}$ inherit the temporal irregularity of $\boldsymbol\xi$. Therefore, we read $\lambda_{i\pm 1/2}\,\dl t$ as the infinitesimal increments $\dl \mu_{i\pm 1/2}(t)$ of $(\mathcal F_t)$-adapted, continuous, real-valued semimartingales $(\mu_{i\pm 1/2}(t))_{t\geq0}$ and interpret \eqref{eq:interpretation_as_SDE_2a} in terms of It\^o integrals w.r.t.\ these semimartingales as
\begin{align}\label{eq:interpretation_as_SDE_2b}
\frac1{\Delta s}\Bigg(
\int_0^\tau\frac{\mathbf r_{i+1}(t)-\mathbf r_i(t)}{\Delta s}\,\dl \mu_{i+1/2}(t)
- \int_0^\tau\frac{\mathbf r_{i}(t)-\mathbf r_{i-1}(t)}{\Delta s}\,\dl \mu_{i-1/2}(t)\Bigg).
\end{align}
\end{subequations}
The integrals exist if the processes $\mathbf r_{i}$ are, for example, adapted and continuous.

To state the resulting constrained stochastic differential system in a compact form, we introduce some notational conventions.

\begin{notation}\label{not:manifold_1}
Concerning the spatially discrete problem the following notations will be used throughout the paper.
\begin{itemize}[leftmargin=7mm]
\item
A generic node-associated element in $\mathbb{R}^{dN}$ is considered as a column vector and denoted by $x=(\mathbf x_i)_{i=1,\ldots,N}$, where $\mathbf x_i\in\mathbb{R}^d$ for $i=1,\ldots,N$. Analogously an edge-associated element in $\mathbb{R}^{N}$ is treated as column vector 
$\xi=(\xi_{i-1/2})_{i=1,\ldots,N}$. 
\item 
The configuration manifold for the inextensible one-sided clamped fiber in the spatially discrete setting is 
\begin{equation}\label{eq:constraint_manifold}
\mathcal M=\big\{x\in\mathbb{R}^{dN}:\,g(x)=0\big\},
\end{equation}
with the constraint function
\begin{equation}\label{eq:g}
g:\mathbb{R}^{dN}\to\mathbb{R}^{N},\quad g=(g_{i-1/2})_{i=1,\dots,N},\qquad 
g_{i-1/2}(x)=\frac12\Big(1-\frac{\|\mathbf x_{i}-\mathbf x_{i-1}\|^2}{(\Delta s)^2}\Big),
\end{equation}
where $\mathbf{x}_0=\mathbf{r}_0$ is fixed due to the clamped boundary condition.
Note that for a free moving fiber the first component of the constraint function cancels out, yielding a larger manifold.
\item
The gradient $\nabla g(x)$ of $g$ at $x\in\mathbb{R}^{dN}$ is considered as a $\mathbb{R}^{dN\times N}$-valued matrix,\footnote{
For $x=(\mathbf x_i)_{i=1,...,N} \in \mathbb{R}^{dN}$ the gradient of the constraint function reads explicitly
\begin{equation*}
\nabla g(x)= \frac{1}{(\Delta s)^2}
\begin{pmatrix}
-(\mathbf x_1-\mathbf r_0) & \;\mathbf x_2-\mathbf x_1 & \mathbf 0 & & &\vdots \\
\mathbf 0 & -(\mathbf x_2-\mathbf x_1) & \;\mathbf x_3-\mathbf x_2 & &  &\vdots \\
\vdots & \mathbf 0 & -(\mathbf x_3-\mathbf x_2) & & & \vdots \\
\vdots & \vdots & \mathbf 0 & \ddots & & \mathbf 0\\
\vdots & \vdots & \vdots & & \ddots &\;\mathbf x_N-\mathbf x_{N-1}\\
\mathbf 0 & \mathbf 0 & \mathbf 0 & \cdots & \cdots &-(\mathbf x_N-\mathbf x_{N-1})
\end{pmatrix}
\;\big(\in\mathbb{R}^{dN\times N}\big)
\end{equation*}
} i.e.,
\begin{equation*}
\nabla g(x)=\big(\nabla g_{1-1/2}(x),\ldots,\nabla g_{N-1/2}(x)\big)
\end{equation*}
where $\nabla g_{i-1/2}(x)\in\mathbb{R}^{dN}$ is a column vector for every $i=1,\ldots,N$. 
\end{itemize}
\end{notation}

\begin{remark}\label{rem:semi_disc_representation_cForce}
Let $r=(r(t))_{t\geq0}$ be the $\mathbb R^{dN}$-valued process for the unknown fiber positions (points) $r(t)=(\mathbf{r}_i(t))_{i=1,...,N}$, and $\mu=(\mu(t))_{t\geq0}$ be the $\mathbb R^{N}$-valued process for the Lagrange multipliers $\mu(t)=(\mu_{i-1/2}(t))_{i=1,...,N}$, then the term in \eqref{eq:interpretation_as_SDE_2b} coincides with the $i$-th component of the $\mathbb R^{dN}$-valued stochastic integral $\int_0^\tau\nabla g(r(t))\,\dl\mu(t)$ for $i=1,\ldots,N-1$.
\end{remark}

After these considerations, the spatially discrete fiber model can be formulated as It\^o-type stochastic differential system with constraint
\begin{equation}\label{eq:intro_semidisc}
\begin{aligned}
 \dl r(t) &= v(t)\,\dl t,\\
 \dl v(t) &= a(t,r(t),v(t))\,\dl t + B(t,r(t),v(t))\,\dl w(t) + \nabla g(r(t)) \,\dl \mu(t),\\
 g(r(t)) &= 0,
 \end{aligned}
\end{equation}
for $t\geq0$, where
\begin{itemize}[leftmargin=7mm]
\item
a solution $(r,v,\mu)$ consists of two $\mathbb R^{dN}$-valued, continuous, adapted processes $r=(r(t))_{t\geq0}$ and $v=(v(t))_{t\geq0}$ as well as a $\mathbb R^{N}$-valued, continuous $(\mathcal F_t)$-semimartingale $\mu=(\mu(t))_{t\geq0}$, see Definition~\ref{def:global_solution} for details concerning the notion of a solution;
\item
$w=(w(t))_{t\geq0}$ is the $\mathbb R^{dN}$-valued Wiener process with coordinate processes $\mathbf w_i=(\mathbf w_i(t))_{t\geq0}$ given by \eqref{eq:w_i};
\item 
the constraint function $g:\mathbb{R}^{dN}\to\mathbb{R}^{N}$ is defined by \eqref{eq:g};
\item
the involved functions for drift $a:\mathbb{R}_0^+\times\mathbb{R}^{dN}\times\mathbb{R}^{dN}\to\mathbb{R}^{dN}$, $a=(\mathbf{a}_i)_{i=1,...,N}$ and diffusion $B:\mathbb{R}_0^+\times\mathbb{R}^{dN}\times\mathbb{R}^{dN}\to\mathbb{R}^{dN\times dN}$, $B=(\mathbf{b}_{i,j})_{i,j=1,...,N}$ are expressed in terms of the discretization stencils\footnote{
The drift and diffusion functions for the fiber model \eqref{eq:intro_semidisc} are
\begin{align*}
\mathbf{a}_i(t,x,y) &= \frac{1}{\alpha^2}\frac{-\mathbf{x}_{i+2}+4\mathbf{x}_{i+1}-6\mathbf{x}_{i}+4\mathbf{x}_{i-1}-\mathbf{x}_{i-2}}{(\Delta s)^{4}}+ \frac{1}{\mathrm{Fr}^{2}}\mathbf{e}_\textrm{g}\\
&\quad + \frac{1}{2\mathrm{Dr}^{2}}\Bigg[\mathbf{C}\Big(\frac{\mathbf{x}_{i+1}+\mathbf{x}_i}{2},\frac{\mathbf x_{i+1}-\mathbf x_i}{\Delta s},t \Big)\Big(\mathbf u\Big(\frac{\mathbf x_{i+1}+\mathbf x_i}2,t\Big)-\frac{\mathbf{y}_{i+1}+\mathbf{y}_i}2\Big)\\
&\hspace*{1.3cm} + \mathbf{C}\Big(\frac{\mathbf{x}_{i}+\mathbf{x}_{i-1}}{2},\frac{\mathbf x_{i}-\mathbf x_{i-1}}{\Delta s},t\Big)\Big(\mathbf{u}\Big(\frac{\mathbf x_{i}+\mathbf x_{i-1}}2,t\Big)-\frac{\mathbf{y}_{i}+\mathbf{y}_{i-1}}2\Big)\Bigg]\\
\mathbf{b}_{i,j}(t,x,y) &= \frac{\beta}{2\sqrt{\Delta s}\mathrm{Dr}^{2}}\,\Bigg[\mathbf{C}\Big(\frac{\mathbf{x}_{i+1}+\mathbf{x}_i}{2},\frac{\mathbf x_{i+1}-\mathbf x_i}{\Delta s},t \Big)\mathbf{D}\Big(\frac{\mathbf{x}_{i+1}+\mathbf{x}_i}{2},\frac{\mathbf{y}_{i+1}+\mathbf{y}_i}{2},\frac{\mathbf x_{i+1}-\mathbf x_i}{\Delta s},t \Big)\\
&\hspace*{1.75cm} +  \mathbf{C}\Big(\frac{\mathbf{x}_{i}+\mathbf{x}_{i-1}}{2},\frac{\mathbf x_{i}-\mathbf x_{i-1}}{\Delta s},t\Big)\mathbf{D}\Big(\frac{\mathbf{x}_{i}+\mathbf{x}_{i-1}}{2},\frac{\mathbf{y}_{i}+\mathbf{y}_{i-1}}{2},\frac{\mathbf x_{i}-\mathbf x_{i-1}}{\Delta s},t\Big)\Bigg]\delta_{i,j}
\end{align*}
for $x=(\mathbf x_i)_{i=1,...,N}, y=(\mathbf y_i)_{i=1,...,N}\in\mathbb R^{dN}$ and with Kronecker delta $\delta_{i,j}$. The incorporation of the boundary conditions is realized by means of the ghost points $\mathbf{x}_j=\mathbf{r}_j$, $j\in\{-1,0,N+1,N+2\}$.
} 
in Sec.~\ref{sec:semi_disc}.
We point out that Assumption~\ref{ass:1} on the drag forces implies that the functions $a$ and $B$ fulfill the local Lipschitz and linear growth conditions \eqref{eq:locallipsch} and \eqref{eq:lineargrowth} below.
\end{itemize}

%%%%%%%%%%%%%%%%%%%%%%%%%%%%%%
%%%
\section{Solution Theory for a Class of SDEs with Holonomic Constraint} 
\label{sec:existence_and_uniqueness}
\setcounter{equation}{0}

This section presents the main theoretical result of the paper. We proof existence and uniqueness of a global solution for a class of SDEs with holonomic constraint that contains our semi-discrete fiber model \eqref{eq:intro_semidisc}. For this purpose we use an explicit representation of the Lagrange multiplier to obtain the existence and uniqueness of a local solution. In a second step, we show that the explosion time of the local solution is infinite almost surely by making use of a one-sided linear growth property of the drift coefficient of the explicitly reformulated equation, which is valid on the constraint manifold.

%%%
\subsection{Setting and main result}

We investigate the class of constrained SDEs
\begin{equation}\label{eq:general_SDE}
\begin{aligned}
 \dl r(t) &= v(t)\,\dl t,\\
 \dl v(t) &= a(t,r(t),v(t))\,\dl t + B(t,r(t),v(t))\,\dl w(t) + \nabla g(r(t)) \,\dl \mu(t),\\
 g(r(t)) &= 0,
 \end{aligned}
\end{equation}
where, as before, $g:\mathbb R^{dN}\to\mathbb R^{N}$ is given by \eqref{eq:g} and $w=(w(t))_{t\geq0}$ is a $\mathbb R^{dN}$-valued standard Wiener process with respect to a normal filtration $(\mathcal F_t)_{t\geq0}$ on the underlying complete probability space $(\Omega,\mathcal F,\mathbb P)$. The drift and diffusion functions $a:\mathbb R^+_0\times\mathbb R^{dN}\times\mathbb R^{dN}\to\mathbb R^{dN}$ and $B:\mathbb R^+_0\times\mathbb R^{dN}\times\mathbb R^{dN} \to\mathbb R^{dN\times dN}$
satisfy the following local Lipschitz and linear growth conditions: For all $T,\,R\in(0,\infty)$ there exist constants $C_{T,R},\,C_T\in[1,\infty)$ such that
\begin{subequations}
\begin{align}\label{eq:locallipsch}
\begin{split}
\|a(t,x,y)-a(t,\tilde x,\tilde y)\|&\leq C_{T,R}\|(x,y)-(\tilde x,\tilde y)\|,\\
\|B(t,x,y)-B(t,\tilde x,\tilde y)\|&\leq C_{T,R}\|(x,y)-(\tilde x,\tilde y)\|,
\end{split} \quad (x,y),\,(\tilde x,\tilde y)\in B_R((0,0)),\;t\in[0,T],
\end{align}
where $B_R((0,0))$ is the open ball in $\mathbb{R}^{dN}\times\mathbb{R}^{dN}$ with radius $R$ and center $(0,0)$, as well as
\begin{align}\label{eq:lineargrowth}
\begin{split}
\|a(t,x,y)\|&\leq C_T(1+\|(x,y)\|),\\
\|B(t,x,y)\|&\leq C_T(1+\|(x,y)\|),
\end{split} \quad (x,y)\in\mathbb{R}^{dN}\times\mathbb{R}^{dN},\;t\in[0,T].
\end{align}
\end{subequations}
We particularly choose the growth constant to be increasing in the end time $T$, i.e., $1\leq C_s \leq C_t$ for $0\leq s\leq t$. Due to Assumption~\ref{ass:1}, the fiber model \eqref{eq:intro_semidisc} fits into this setting. 
Note that as convenient notation throughout the paper we use $\|\cdot\|$ for the Euclidean norm and $\langle\cdot,\cdot\rangle$ for the inner product in possibly different spaces such as $\mathbb R^{dN}$, $\mathbb R^{dN}\times \mathbb R^{dN}$ and $\mathbb R^{dN\times dN}$. 

Let $T_x\mathcal M$ be the tangent space at the constraint manifold $\mathcal M$ in a point $x\in \mathcal M$, compare Notation~\ref{not:manifold_2}. The notion of a global solution to \eqref{eq:general_SDE} can be made precise as follows.

\begin{definition}[Global solution]\label{def:global_solution}
A (global) solution  to the constrained SDE \eqref{eq:general_SDE} with initial conditions $r_0\in \mathcal M$ and $v_0\in T_{r_0}\mathcal M$, is a triple $(r,v,\mu)$ consisting of $\mathbb R^{dN}$-valued, continuous, $(\mathcal F_t)$-adapted processes $r=(r(t))_{t\geq0}$ and $v=(v(t))_{t\geq0}$ as well as a $\mathbb R^{N}$-valued, continuous $(\mathcal F_t)$-semimartingale $\mu=(\mu(t))_{t\geq0}$ with $\mu(0)=0$ such that, $\mathbb P$-almost surely, the following equalities hold for all $t\geq0$,
\begin{align*}
r(t)&= r_0+\int_0^t v(u)\,\dl u\\
v(t) &= v_0+\int_0^t a(u,r(u),v(u))\,\dl u + \int_0^t B(u,r(u),v(u))\,\dl w(u) + \int_0^t\nabla g(r(u)) \,\dl \mu(u),\\
g(r(t))&=0.
\end{align*}
\end{definition}
 
\begin{notation}\label{not:manifold_2} Associated with the constraint manifold $\mathcal M$ in \eqref{eq:constraint_manifold}  we use the following notations.
\begin{itemize}[leftmargin=7mm]
\item
The Gram matrix associated with the constraint is given by
\begin{equation*}
G(x)=[\nabla g(x)]^\top\nabla g(x)\;\in \mathbb{R}^{N\times N}.
\end{equation*}
It is defined for all $x\in\mathbb{R}^{dN}$ and is invertible for all $x$ close enough to the manifold $\mathcal M$, e.g., for all $x\in\mathcal  M^\varepsilon:=\{y\in \mathbb{R}^{dN}:\operatorname{dist}(y,\mathcal M)<\varepsilon\}$ with $\varepsilon=\varepsilon(N)>0$ small enough, see Lemma \ref{lem:inverse_G} for details.
\item
For  all $x\in\mathbb{R}^{dN}$ such that $G(x)$ is invertible we introduce the matrix
\begin{equation*}
P(x):=\operatorname{Id} -\nabla g(x) G^{-1}(x)[\nabla g(x)]^\top\;\in\mathbb{R}^{dN\times dN}.
\end{equation*}
If $x\in\mathcal M$, $P(x)$ represents the orthogonal projection of $\mathbb{R}^{dN}$ onto the tangent space $T_x\mathcal M=\{y\in\mathbb{R}^{dN}:[\nabla g(x)]^\top y=0\}$ at $\mathcal M$ in $x$.
\item 
By $D^2 g(x)$ we denote the second derivative of $g$ at $x\in\mathbb{R}^{dN}$. It is a bilinear mapping $D^2 g(x):\mathbb{R}^{dN}\times\mathbb{R}^{dN} \to \mathbb{R}^{N}$, i.e.,
\begin{equation*}
D^2 g(x)(v,w)=\big(v^\top D^2 g_{i-1/2}(x)w\big)_{i=1,\ldots,N}
\end{equation*}
for $v,w\in\mathbb{R}^{dN}$, where $D^2 g_{i-1/2}(x)\in\mathbb{R}^{dN\times dN}$ is the Hessian matrix of $g_{i-1/2}$ at $x$.
\end{itemize}
\end{notation}

\begin{theorem}[Existence and uniqueness of a global solution]\label{thm:main_result}
Assume that the local Lipschitz and linear growth conditions \eqref{eq:locallipsch}-\eqref{eq:lineargrowth} hold, and let $r_0\in \mathcal M$, $v_0\in T_{r_0}\mathcal M$. Then, there exists a unique (up to indistinguishability) global solution $(r,v,\mu)$ to the constrained SDE \eqref{eq:general_SDE} with initial conditions $r_0,v_0$, and, $\mathbb P$-almost surely, the following equality holds for all $t\geq 0$
\begin{equation}\label{eq:mu}
\begin{aligned}
\mu(t)&= -\int_0^tG^{-1}(r(u))\Big\{[\nabla g(r(u))]^\top\Big(a(u,r(u),v(u))\,\dl u + B(t,r(u),v(u))\,\dl w(u)\Big)\\
&\quad+D^2 g(r(u))\big(v(u),v(u)\big)\,\dl u\Big\}.
\end{aligned}
\end{equation}
Moreover, there exists a finite constant $C>0$ which does not depend on $T$ such that the global solution satisfies
\begin{align}\label{ineq:sup}
\mathbb E\bigl(\sup_{t\in[0,T]}\|(r(t),v(t))\|^4\bigr) \leq C\bigl(\mathbb E\|(r(0),v(0))\|^4+ C_T^2T^3\bigr)\operatorname{exp}\bigl(CC_T^2T^3\bigr)
\end{align}
for all $T\geq 1$ with the linear growth constant $C_T\geq1$ of \eqref{eq:lineargrowth}. 
\end{theorem}

The equality \eqref{eq:mu} for the Lagrange multiplier implies that we can rewrite the original constrained SDE \eqref{eq:general_SDE} into the following explicit form,
\begin{equation}\label{eq:explicit_SDE}
\begin{aligned}
 \dl r(t) &= v(t)\,\dl t,\\
 \dl v(t) &= P(r(t))\bigl[a(t,r(t),v(t))\,\dl t+B(t,r(t),v(t))\,\dl w(t)\bigr] \\
          &\quad
          +\nabla g(r(t))G^{-1}(r(t))D^2(g(r(t)))(v(t),v(t))\,\dl t.
 \end{aligned}
\end{equation}
In the subsequent solution theory we will actually see that the existence of a (global) solution $(r,v,\mu)$ to \eqref{eq:general_SDE} and the existence of a (global) solution $(r,v)$ to \eqref{eq:explicit_SDE} are equivalent. In particular, if $(r,v)$ is a (global) solution to \eqref{eq:explicit_SDE} and $\mu$ is defined via \eqref{eq:mu}, then $(r,v,\mu)$ is a solution to \eqref{eq:general_SDE}.
The explicit version \eqref{eq:explicit_SDE} of the constrained SDE \eqref{eq:general_SDE} is known in principle, see, e.g., \cite{ciccotti:p:2006,kallemov:p:2011,lelievre:p:2012}. However, as the possibility of explosion in finite time has not been taken into account in the mentioned papers and the references therein, the presentation of a rigorous derivation seems to be justified. 

%%%
\subsection{Local solvability}
\label{sec:local_solvability}

Let us briefly recall the concepts of local solvability and explosion times. Given a $(\mathcal F_t)$-stopping time $\sigma:\Omega\to(0,\infty]$, the stochastic interval $[\![0,\sigma)\!)$ is the subset of $\Omega\times\mathbb{R}^+_0$ defined by
\begin{equation*}
[\![0,\sigma)\!):=\big\{(\omega,t)\in \Omega\times\mathbb R^+_0: t\in [0,\sigma(\omega))\big\}.
\end{equation*}

\begin{definition}[Local solution, explosion time]\label{definition:localsolution} Let $\sigma:\Omega\to(0,\infty]$ be a $(\mathcal F_t)$-stopping time. A \emph{local solution} to the constrained SDE \eqref{eq:general_SDE} up to time $\sigma$ with initial conditions $r_0\in\mathcal M$, $v_0\in T_{r_0}\mathcal M$ is a mapping
 $$(r,v,\mu):[\![0,\sigma)\!)\to \mathbb R^{dN}\times\mathbb R^{dN}\times \mathbb R^{N},$$
 measurable w.r.t.\ the trace $\sigma$-algebra $(\mathcal F\otimes\mathcal B(\mathbb{R}^+_0))\cap[\![0,\sigma)\!)$, such that, for all stopping times $\tilde \sigma <\sigma$, the triple $(\tilde r,\tilde v,\tilde\mu)$ consisting of the stopped process $\tilde r=(\tilde r(t))_{t\geq0}:=(r(t\wedge \tilde \sigma))_{t\geq0}$, $\tilde v=(\tilde v(t))_{t\geq0}:=(v(t\wedge \tilde \sigma))_{t\geq0}$ and $\tilde \mu=(\tilde \mu(t))_{t\geq0}:=(\mu(t\wedge \tilde \sigma))_{t\geq0}$ is a global solution with initial conditions $r_0,v_0$ to 
 \begin{equation*}\label{eq:stopped_SDE}
\begin{aligned}
 \dl\tilde r(t) &=\mathbbm{1}_{[0,\tilde\sigma]}(t)\,\tilde v(t)\,\dl t,\\
 \dl\tilde v(t) &=\mathbbm{1}_{[0,\tilde\sigma]}(t)\,\big[a(t,\tilde r(t),\tilde v(t))\,\dl t + B(t,\tilde r(t),\tilde v(t))\,\dl w(t) + \nabla g(\tilde r(t)) \,\dl\tilde\mu(t)\big],\\
 g(\tilde r(t)) &= 0.
 \end{aligned}\
\end{equation*}
We call $\sigma$ the \emph{explosion time} of $(r,v,\mu)$ if $\sigma$ is $\mathbb P$-a.s.\ equal to the limit of the increasing sequence  of stopping times $(\sigma_K)_{K\in\mathbb N}$ defined by
\begin{align*}
\sigma_K=\inf\{t\in[0,\infty]:\|(r(t),v(t),\mu(t))\|^2:=\|r(t)\|^2+\|v(t)\|^2+\|\mu(t)\|^2\geq K^2\},
\end{align*}
where we set $\inf\emptyset:=\infty$.
\end{definition}

The local solution to the explicit equation~\eqref{eq:explicit_SDE} and its explosion time are defined analogously to Definition~\ref{definition:localsolution}, cf.\ \cite{hsu:b:2002}. The following technical lemma is needed to prove the local solvability.

\begin{lemma}\label{lem:inverse_G}
For all $N\in \mathbb N$ let $\varepsilon=\varepsilon(N)=\Delta s/4$. Then, for all $x\in\mathcal M^{\varepsilon}$, the Gram matrix $G(x)$ is invertible and $\inf_{x\in\mathcal M^{\varepsilon}}\operatorname{det}(G(x))>0$.
\end{lemma}
\begin{proof}
By Cramer's rule we know that $G(x)$ is invertible with $G^{-1}(x)=\operatorname{adj}(G(x))/\operatorname{det}(G(x))$ if $\operatorname{det}G(x) >0$. Fix some $x=(\mathbf x_i)_{i=1,...,N}\in\mathcal M^\varepsilon$. The explicit representation of $\nabla g(x)$ yields that $G(x)=\big( G_{i,j}(x)\big)_{i,j=1}^{N}$ can be defined as a tridiagonal matrix whose elements fulfill
\begin{align*}
G_{1,1}(x)=(\Delta s)^{-4}\|\mathbf x_{1}- \mathbf r_{0}\|^2,& \qquad  G_{1,2}(x)= G_{2,1}(x)= -(\Delta s)^{-4}\big\langle\mathbf x_1-\mathbf r_0,\mathbf x_{2}-\mathbf x_1\big\rangle\\
G_{i,i}(x)=2(\Delta s)^{-4}\|\mathbf x_{i}- \mathbf x_{i-1}\|^2,& \qquad  G_{j,j+1}(x)= G_{j+1,j}(x)= -(\Delta s)^{-4}\big\langle\mathbf x_j-\mathbf x_{j-1},\mathbf x_{j+1}-\mathbf x_j\big\rangle
\end{align*}
for all $i=2,\dots,N$ and $j=2,\dots,N-1$ with inner product $\langle \cdot, \cdot \rangle$. The triangle inequality applied to some element of $\mathcal M\cap \overline{B_{\varepsilon}(x)}$ implies $\Delta s- 2\varepsilon\leq \|\mathbf x_{i}-\mathbf x_{i-1}\|$ and $\Delta s- 2\varepsilon\leq\|\mathbf x_{1}-\mathbf r_{0}\|$.
The multilinearity of the determinant leads to
\begin{align*}
\operatorname{det}(G(x))&\geq \operatorname{det}(\tilde G(x)) \left(\frac{(\Delta s- 2\varepsilon)^2}{(\Delta s)^4}\right)^{N} = \operatorname{det}(\tilde G(x)) (2\Delta s)^{-2N}
\end{align*}
for the tridiagonal matrix $\tilde G(x)=\big(\tilde G_{i,j}(x)\big)_{i,j=1}^{N}$ defined by
\begin{align*}
\tilde G_{1,1}(x)=1, &\qquad \tilde G_{1,2}(x)=\tilde G_{2,1}(x)= -\frac{\big\langle\mathbf x_1-\mathbf r_{0},\mathbf x_{2}-\mathbf x_1\big\rangle}{\|\mathbf x_{2}-\mathbf x_1\|\|\mathbf x_{1}-\mathbf r_0\|}\in[-1,1]\\
\tilde G_{i,i}(x)=2, &\qquad \tilde G_{j,j+1}(x)=\tilde G_{j+1,j}(x)= -\frac{\big\langle\mathbf x_j-\mathbf x_{j-1},\mathbf x_{j+1}-\mathbf x_j\big\rangle}{\|\mathbf x_{j+1}-\mathbf x_j\|\|\mathbf x_{j-1}-\mathbf x_j\|}\in[-1,1].
\end{align*}
By $\tilde G^{(k)}(x)=\big(\tilde G_{i,j}(x)\big)_{i,j=1}^{k}$ we denote the $k\times k$ upper left block of $\tilde G(x)$. The Laplace expansion yields the following recursive formula for $\tilde G^{(k)}(x)$ w.r.t.\ $k$
\begin{align*}
\operatorname{det} \tilde G^{(k)}(x)=2\operatorname{det} \tilde G^{(k-1)}(x)- G_{k,k-1}(x)^2\operatorname{det} \tilde G^{(k-2)}(x).
\end{align*}
By induction $\operatorname{det} \tilde G^{(k)}(x)\geq \operatorname{det} \tilde G^{(k-1)}(x)$ can be concluded, since $\operatorname{det} (\tilde G^{(2)}(x))\geq\operatorname{det}( \tilde G^{(1)}(x))$ and
\begin{align*}
\operatorname{det} \tilde G^{(k)}(x)&=2\operatorname{det} \tilde G^{(k-1)}(x)- G_{k,k-1}(x)^2\operatorname{det} \tilde G^{(k-2)}(x)\geq 2\operatorname{det} \tilde G^{(k-1)}(x)- \operatorname{det} \tilde G^{(k-2)}(x)\\
&\geq\operatorname{det} \tilde G^{(k-1)}(x).
\end{align*}
Therefore, we have $\operatorname{det}(\tilde G(x))\geq \operatorname{det}(\tilde G^{(1)}(x))=1$ and $$\operatorname{det}(G(x))\geq (2\Delta s)^{-2N}$$ is bounded from below for any $x\in\mathcal M^\varepsilon$.
\end{proof}

\begin{lemma} \label{lem:unique_mu}
Let $\sigma:\Omega\to(0,\infty]$ be a $(\mathcal F_t)$-stopping time. 
If $(r,v,\mu)$ is a local solution to \eqref{eq:general_SDE} up to time $\sigma$ with initial conditions $r_0\in\mathcal M$, $v_0\in T_{r_0}\mathcal M$, then the equality \eqref{eq:mu} holds on $[\![0,\sigma)\!)$ in the following sense: For every $(\mathcal F_t)$-stopping time $\tilde\sigma<\sigma$ we have that, $\mathbb{P}$-almost surely, the following equality holds for all $t\geq0$,
\begin{equation}\label{eq:tildemu}
\begin{aligned}
\tilde\mu(t)&= -\int_0^t\mathbbm{1}_{[0,\tilde\sigma]}(u)G^{-1}(\tilde r(u))\Big\{[\nabla g(\tilde r(u))]^\top\Big(a(t,\tilde r(u),\tilde v(u))\,\dl u + B(t,\tilde r(u),\tilde v(u))\,\dl w(u)\Big)\\
&\quad+D^2 g(\tilde r(u))\big(\tilde v(u),\tilde v(u)\big)\,\dl u\Big\},
\end{aligned}
\end{equation}
where $\tilde r$, $\tilde v$ and $\tilde\mu$ are as in Definition~\ref{definition:localsolution}. Conversely, if $(r,v)$ is a local solution to the explicit equation~\eqref{eq:explicit_SDE} up to time $\sigma$ with initial conditions $r_0\in\mathcal M$, $v_0\in T_{r_0}\mathcal M$, and $\mu:[\![0,\sigma)\!)\to\mathbb{R}^{N}$ is defined by $\mu|_{[\![0,\tilde \sigma)\!)}:=\tilde \mu|_{[\![0,\tilde \sigma)\!)}$ with $\tilde\mu$ given by \eqref{eq:tildemu}, then $(r,v,\mu)$ is a local solution to \eqref{eq:general_SDE} up to time $\sigma$ with initial conditions $r_0,v_0$.
\end{lemma}

\begin{proof}
Let $(r,v,\mu)$ be a local solution to \eqref{eq:general_SDE} up to time $\sigma$ with initial conditions $r_0\in\mathcal M$ and $v_0\in T_{r_0}\mathcal M$. Furthermore let $\tilde \sigma$ be any fixed stopping time fulfilling $\tilde\sigma <\sigma$. Note that, $\mathbb P$-almost surely, $g\circ r$ is differentiable and constantly zero on $[0,\tilde \sigma]$. Thus, applying the chain rule on the constraint yields, $\mathbb P$-almost surely,
\begin{align*}
0=\frac{\dl}{\dl t} g(r(t))=(\nabla g(r(t)))^\top v(t) \qquad \text{for all } t\in [0,\tilde \sigma].
\end{align*}
Recall $\tilde r=(r(t\wedge \tilde\sigma))_{t\geq 0}$ and $\tilde v=(v(t\wedge \tilde\sigma))_{t\geq 0}$ from Definition \ref{definition:localsolution}. Since $r$, $\tilde r$ as well as $v$, $\tilde v$ coincide on $[\![0,\tilde \sigma]\!]$, it holds, $\mathbb{P}$-almost surely, that
\begin{align}\label{eq:hidden_Dg}
0=(\nabla g(\tilde r(t)))^\top \tilde v(t) \qquad \text{for all } t\in[0,\infty).
\end{align}
Let $f:\mathbb R^{dN}\times \mathbb R^{dN} \to\mathbb R$ be defined by $f(x,y)=(\nabla g(x))^\top y$ for any $x,y \in \mathbb R^{dN}$, and let $\nabla_x$ and $\nabla_y$ be the gradients w.r.t.\ the variables $x$ and $y$, respectively, i.e., $\nabla_x f(x,y)=\nabla(f(\cdot,y))(x)$. We want to apply It\^o's formula to $f(\tilde r,\tilde v)$. Note that the process $\tilde r$ has bounded variation and therefore regarded as a semimartingale its martingale part is zero. Then, by It\^o's formula,
\begin{align}\label{eq:itoonrv1}
f(\tilde r(t),\tilde v(t))-f(r_0,v_0)&=\int_0^t \nabla_x f(\tilde r(u),\tilde v(u))\dl \tilde r(u)+\int_0^t \nabla_y f(\tilde r(u),\tilde v(u)) \dl \tilde v(u)
\end{align}
where all second order terms vanish since $(\tilde r(t))_{t\geq 0}$ has bounded variation and $f$ is linear in $y$. Combining \eqref{eq:hidden_Dg} and \eqref{eq:itoonrv1} and inserting the definition of $f$, $\tilde r$ and $\tilde v$ yields
\begin{align}\label{eq:itoonrv2}
\begin{split}
0&=\int_0^t \mathbbm{1}_{[0,\tilde \sigma]}(u)\biggl\{(\nabla g(\tilde r(u)))^\top\Bigl[ (\nabla g(\tilde r(u)))\dl \tilde \mu(u)+ a(u,\tilde r(u),\tilde v(u))\dl u\\
&\qquad+ B(u,\tilde r(u),\tilde v(u))\dl w(u)\Bigr]+D^2g(\tilde r(u))(\tilde v(u), \tilde v(u))\dl u\biggr\}.
\end{split}
\end{align}
Integrating $G^{-1}(r(\cdot))$ w.r.t.\ the semimartingale defined by the right hand side of \eqref{eq:itoonrv2} results in \eqref{eq:tildemu}.
Let $(r,v)$ be any local solution to the explicit equation~\eqref{eq:explicit_SDE} up to time $\sigma$ with initial conditions $r_0\in \mathcal M$ and $v_0\in T_{r_0}\mathcal M$. Further let $\tilde \sigma$ be again any fixed stopping time fulfilling $\tilde\sigma <\sigma$. Note that, $\mathbb{P}$-almost surely, $g\circ r$ is differentiable on $[0,\tilde \sigma]$. Recall that $g(r_0)=0$. Therefore $(r,v)$ combined with $\mu:[\![0,\sigma)\!)\to\mathbb{R}^{N}$ defined by \eqref{eq:tildemu} is a local solution to \eqref{eq:general_SDE} up to time $\sigma$ with initial conditions $r_0,v_0$, if, $\mathbb{P}$-almost surely,
\begin{align*}
0=\frac{\dl}{\dl t} g(r(t))=(\nabla g(r(t)))^\top v(t) \qquad \text{for all } t\in [0,\tilde \sigma].
\end{align*}
This equality can be verified by using It\^o's formula.
\end{proof}

The unique local solvability of the manifold-valued explicit SDE~\eqref{eq:explicit_SDE} can be shown by using the general (abstract) solution theory in \cite{hsu:b:2002}. For the sake of completeness and in order to gain more insight in our specific problem, we state here a concise proof that makes use of the structure of $\mathcal M$ and the explicit equation~\eqref{eq:explicit_SDE}.

\begin{lemma}\label{lem:local_solv}
The SDE \eqref{eq:explicit_SDE} has a unique local solution $(r,v)$ up to its explosion time $\sigma$.
\end{lemma}

\begin{proof}
We consider a locally Lipschitz continuous extension of the coefficients in the explicit SDE \eqref{eq:explicit_SDE} towards the space $\mathbb R^{dN}$. Recall that the drift and diffusion functions $a$ and $B$ are locally Lipschitz continuous and already defined on $\mathbb R^{dN}$. Moreover, $\nabla g$ and $D^2g$ are defined on $\mathbb R^{dN}$ and are smooth. According to Lemma~\ref{lem:inverse_G}, $G(x)$ has an inverse $G^{-1}(x)$ whose matrix norm is bounded for all $x\in\mathcal M^{\varepsilon}$. We fix $\varepsilon>0$ and introduce the extension
\begin{align}\label{eq:extensionG}
\tilde G^{-1}(x)=\begin{cases}
\varphi(\operatorname{dist}(x,\mathcal M)/\varepsilon)\,G^{-1}(x) &\text{for all }x\in\mathcal M^{\varepsilon}\\
0 &\text{else}
\end{cases}
\end{align}
with $\varphi\in C^1\bigl([0,\infty)\bigr)$ such that $\varphi(x)=1$ for all $x\in[0,1/2]$ and $\varphi(x)=0$ for all $x\in[3/4,\infty)$. The fact that $G^{-1}(x)=\operatorname{adj}(G(x))/\operatorname{det}(G(x))$ and $\inf_{x\in\mathcal M^{\varepsilon}}\operatorname{det}(G(x))>0$ (cf.\ Lemma~\ref{lem:inverse_G}) implies that $G^{-1}(x)|_{\mathcal M^\varepsilon}$ can be written as a composition of smooth functions on $\mathcal M^\varepsilon$. Thus, we have the following explicit stochastic differential system with extended locally Lipschitz coefficients
\begin{align}\label{eq:extendedexplicitSDE}
\begin{aligned}
 \dl r(t) &= v(t)\,\dl t,\\
 \dl v(t) &= \tilde P(r(t))\bigl[a(t,r(t),v(t))\,\dl t+B(t,r(t),v(t))\,\dl w(t)\bigr]\\
          &\quad-\nabla g(r(t))\tilde G^{-1}(r(t))D^2(g(r(t)))(v(t),v(t))\,\dl t.
 \end{aligned}
\end{align}
The existence and uniqueness of a local solution $(r,v)$ to the SDE~\eqref{eq:extendedexplicitSDE} up to its explosion time $\sigma$ is covered in literature, see e.g.\ \cite[Lemma 18.15]{schilling:b:2014}. In addition, $\mathbb P$-almost surely, $r(t)\in\mathcal M$ and $v(t)\in T_{r(t)}\mathcal M$ for all $t\in[0,\sigma)$, which can be concluded from applying It\^o's formula to the derivative of $g\circ r$ (as done in the proof of Lemma \ref{lem:unique_mu}). Consequently, the local solution of \eqref{eq:extendedexplicitSDE} depends neither on $\varepsilon$ nor on the extension~\eqref{eq:extensionG}. This guarantees that the process $(r,v)$ is the local solution to \eqref{eq:explicit_SDE}.
\end{proof}

Lemma~\ref{lem:unique_mu} and Lemma~\ref{lem:local_solv} provide the existence and uniqueness of a local solution $(r,v,\mu)$ to the constrained SDE~\eqref{eq:general_SDE} up to the stopping time $\sigma$ appearing in Lemma~\ref{lem:local_solv}. The definition of the explosion time of $(r,v,\mu)$ and the explicit representation of $\mu$ in \eqref{eq:tildemu} yield that $\sigma$ is the explosion time of this solution in the sense of Definition~\ref{definition:localsolution}.

%%%
\subsection{Global solvability}
\label{sec:global_solvability}

We prove the global existence and uniqueness result (Theorem~\ref{thm:main_result}) by showing that the explosion time $\sigma$ of the local solution satisfies $\mathbb P(\sigma=\infty)=1$ and that estimate \eqref{ineq:sup} holds.

\begin{proposition}\label{globalsolution}
Let $(r,v)$ be the unique local solution of the explicit SDE~\eqref{eq:explicit_SDE} and let $\sigma:\Omega\to(0,\infty]$ be its explosion time. Then we have $\mathbb P(\sigma=\infty)=1$, i.e., $(r,v)$ is a global solution to the SDE~\eqref{eq:explicit_SDE}. Furthermore, the global solution to \eqref{eq:explicit_SDE} is unique up to indistinguishability and satisfies the estimate \eqref{ineq:sup}.
\end{proposition}
\begin{proof}
Unless otherwise stated, $C>0$ denotes a finite constant which may change its value with every new appearance.
Let $T\geq 1$ and let the sequence of stopping times $(\sigma_K)_{K\in \mathbb N}$ be defined similarly as in Definition~\ref{definition:localsolution} by
\begin{align*}
\sigma_K=\inf\{t\in [0,\infty):\|(r(t)\|^2+\|v(t))\|^2\geq K^2\}
\end{align*}
where $\inf\emptyset=\infty$. The stopping time approximates the explosion time of $(r,v)$ from below, i.e., $\sigma_K<\sigma$ and $\sigma_K\to \sigma$ as $K\rightarrow \infty$ $\mathbb P$-almost surely. Consider 
$K\in \mathbb N$ fixed. Applying It\^o's formula, see, e.g., \cite[p.\ 243]{schilling:b:2014}, to $\|(r(t\wedge \sigma_K),v(t\wedge \sigma_K))\|^2$ yields
\begin{align*}
 \|(r(t\wedge \sigma_K),v(t\wedge \sigma_K))\|^2&=\|(r(0),v(0))\|^2+\int_0^t\mathbbm{1}_{[0,\sigma_K]}(u) 2\big\langle r(u),v(u)\big\rangle\dl u\\
 &+\int_0^t\mathbbm{1}_{[0,\sigma_K]}(u) 2\big\langle v(u),P(r(u))a(u,r(u),v(u))\big\rangle\dl u\\
 &+\int_0^t\mathbbm{1}_{[0,\sigma_K]}(u) 2\big\langle v(u),P(r(u))B(u,r(u),v(u))[\cdot]\big\rangle\dl w(u)\\
 &+\int_0^t\mathbbm{1}_{[0,\sigma_K]}(u) 2\big\langle v(u),\nabla g(r(u)) G^{-1}(r(u))D^2g(r(u))(v(u),v(u))\big\rangle\dl u\\
 &+\int_0^t\mathbbm{1}_{[0,\sigma_K]}(u) \operatorname{Tr}\bigl((P(r(u))B(u,r(u),v(u)))^\top(P(r(u))B(u,r(u),v(u)))\bigr)\dl u
\end{align*}
where $\operatorname{Tr}:\mathbb R^{dN\times dN}\to \mathbb R$ is the trace operator. Since $(r,v)$ is a local solution to the explicit SDE \eqref{eq:explicit_SDE}, we have, $\mathbb P$-almost surely, $v(u)\in T_{r(u)}\mathcal M$ for all $u<\sigma$. This yields, $\mathbb{P}$-almost surely, for all $u<\sigma$,
\begin{align}\label{eq:onesidedlineargrowth}
\begin{split}
\big\langle v(u),\nabla g(r(u))& G^{-1}(r(u))D^2g(r(u))(v(u),v(u))\big\rangle\\
&=\big\langle [\nabla g(r(u))]^\top v(u), G^{-1}(r(u))D^2g(r(u))(v(u),v(u))\big\rangle\\
&=\big\langle 0,G^{-1}(r(u))D^2g(r(u))(v(u),v(u))\big\rangle\\
&=0
\end{split}
\end{align}
and $\langle v(u),P(r(u)) y\rangle=\langle v(u),y\rangle$ for all $y\in\mathbb R^{dN}.$
Inserting these simplifications above and using the estimate $(\sum_{i=1}^{n}\phi_i)^2\leq n \sum_{i=1}^n \phi_i^2$ we obtain
\begin{align*}
 \mathbb E\sup_{t\leq T}\|(r(t\wedge \sigma_K),v(t\wedge \sigma_K))\|^4&\leq C\biggl( \mathbb E\|(r(0),v(0))\|^4+\mathbb E\sup_{t\leq T}\Bigl(\int_0^t\mathbbm{1}_{[0,\sigma_K]}(u) \big\langle r(u),v(u)\big\rangle\dl u\Bigr)^2\\
 &\quad+\mathbb E\sup_{t\leq T}\Bigl(\int_0^t\mathbbm{1}_{[0,\sigma_K]}(u) \big\langle v(u),a(u,r(u),v(u))\big\rangle\dl u\Bigr)^2\\
 &\quad+\mathbb E\sup_{t\leq T}\Bigl(\int_0^t\mathbbm{1}_{[0,\sigma_K]}(u) \big\langle v(u),B(u,r(u),v(u))[\cdot]\big\rangle\dl w(u)\Bigr)^2\\
 &\quad+\mathbb E\sup_{t\leq T}\Bigl(\int_0^t\mathbbm{1}_{[0,\sigma_K]}(u) \|P(r(u))B(u,r(u),v(u)))\|^2\dl u\Bigr)^2\biggr).
\end{align*}
The occurring scalar products and the stochastic integral can be estimated using the Cauchy-Schwarz inequality in $\mathbb R^{dN}$ and Doob's maximal inequality combined with It\^o's isometry.
Furthermore, for each $x\in M$, $P(x)$ is an orthogonal projection onto a linear subspace, in particular $\|P(x)\|\leq 1$. Applying also the Cauchy-Schwarz inequality for integrals, we obtain
\begin{align*}
 \mathbb E\sup_{t\leq T}\|(r(t\wedge \sigma_K),v(t\wedge \sigma_K))\|^4&\leq C\biggl( \mathbb E\|(r(0),v(0))\|^4+\mathbb E\sup_{t\leq T}\Bigl(T^2\int_0^t\mathbbm{1}_{[0,\sigma_K]}(u) \| r(u)\|^2\|v(u)\|^2\dl u\Bigr)\\
 &\quad+\mathbb E\sup_{t\leq T}\Bigl(T^2\int_0^t\mathbbm{1}_{[0,\sigma_K]}(u) \|v(u)\|^2\|a(u,r(u),v(u))\|^2\dl u\Bigr)\\
 &\quad+\mathbb E\Bigl(\int_0^T\mathbbm{1}_{[0,\sigma_K]}(u) \|v(u)\|^2\|B(u,r(u),v(u))\|^2\dl u\Bigr)\\
 &\quad+\mathbb E\sup_{t\leq T}\Bigl(T^2\int_0^t\mathbbm{1}_{[0,\sigma_K]}(u) \|B(u,r(u),v(u)))\|^4\dl u\Bigr)\biggr).
\end{align*}
As the integrands in the integrals w.r.t\ $\dl u$ are non-negative we enlarge the domain of integration to $[0,T]$. Applying the linear growth conditions~\eqref{eq:lineargrowth} leads to
\begin{align*}
 \mathbb E\sup_{t\leq T}\|(r(t\wedge \sigma_K),v(t\wedge \sigma_K))\|^4&\leq C\biggl( \mathbb E\|(r(0),v(0))\|^4+\mathbb E\Bigl( T^2\int_0^T\mathbbm{1}_{[0,\sigma_K]}(u) \| r(u)\|^2\|v(u)\|^2\dl u\Bigr)\\
 &\quad+\mathbb E\Bigl(C_T(T^2+1)\int_0^T\mathbbm{1}_{[0,\sigma_K]}(u) \|v(u)\|^2\bigl(1+\|r(u)\|^2+\|v(u)\|^2\bigr)\dl u\Bigr)\\
 &\quad+\mathbb E\Bigl(C_T^2T^2\int_0^T\mathbbm{1}_{[0,\sigma_K]}(u) \bigl(1+\|r(u)\|^2+\|v(u)\|^2\bigr)^2\dl u\Bigr)\biggr).
\end{align*} 
Note that $\mathbbm 1_{[0,\sigma_K]}(u)\|(r(u),v(u))\|^4\leq\|(r(u\wedge\sigma_K),v(u\wedge\sigma_K))\|^4$ and $\phi\leq \phi^2+1$ for arbitrary $\phi$.  Consequently, as $T\geq 1$ and $C_T\geq 1$, the estimate can be simplified to 
\begin{align*}
 \mathbb E\sup_{t\leq T}\|(r(t\wedge \sigma_K),v(t\wedge \sigma_K))\|^4&\leq C\biggl(\mathbb E\|(r(0),v(0))\|^4+C_T^2 T^3\\
&\quad+C_T^2 T^2\mathbb E\Bigl(\int_0^T \|(r(u\wedge\sigma_K),v(u\wedge\sigma_K))\|^4\dl u\Bigr)\biggr).
\end{align*} 
Moreover, $\mathbb E \int_0^T \|(r(u\wedge\sigma_K),v(u\wedge\sigma_K))\|^4\dl u \leq \int_0^T \mathbb E\sup_{u\leq T}\|(r(u\wedge\sigma_K),v(u\wedge\sigma_K))\|^4\dl t$ by Fubini's theorem. Therefore, by means of Gronwall's lemma we obtain
\begin{align}\label{ineq:estimatesupnorm}
 \mathbb E\sup_{t\leq T}\|(r(t\wedge \sigma_K),v(t\wedge \sigma_K))\|^4
&\leq C\bigl(\mathbb E\|(r(0),v(0))\|^4+ C_T^2 T^3\bigr)\operatorname{exp}\bigl(CC_T^2 T^3\bigr)=:(C_T')^2.
\end{align}
Note that the right hand side in the estimate \eqref{ineq:estimatesupnorm} is independent of $K$. By H\"older's inequality we obtain 
\begin{align}\label{ineq:estimatenormsimple}
\mathbb E\|(r(T\wedge \sigma_K),v(T\wedge \sigma_K))\|^2\leq\bigl(\mathbb E\|(r(T\wedge \sigma_K),v(T\wedge \sigma_K))\|^4\bigr)^{\frac{1}{2}}\leq C_T'.
\end{align}
Finally, $\mathbb P(\sigma=\infty)=1$ can be shown by using \eqref{ineq:estimatenormsimple} in the following setting:
Let $(\gamma_T)_{T\geq 0}$ be the family of random variables $\gamma_T:\Omega \to [0,\infty]$ defined by
\begin{align*}
 \gamma_T(\omega)=\begin{cases}
                   \|(r(T,\omega),v(T,\omega))\|^2 &\text{ for }T<\sigma(\omega)\\
		   \infty &\text{ else.}
                  \end{cases}
\end{align*} The measurability of these random variables follows directly from the measurability of $\sigma$ and of $\|(r(T),v(T))\|^2.$
Obviously, we have $\{\sigma\leq T\}=\{\gamma_T=\infty\}$.
The definition of $\sigma_K$ yields the following estimate for every $t<K$,
\begin{align*}
\gamma_T\wedge K\leq \|(r(T\wedge \sigma_K),v(T\wedge \sigma_K))\|^2,
\end{align*}
where $\infty\wedge K=K$. Consequently, 
\begin{align*}
 \mathbb E(\gamma_T)&= \mathbb E\Bigl(\lim_{K\to\infty}(\gamma_T\wedge K)\Bigr)=\lim_{K\to\infty}\mathbb E(\gamma_T\wedge K) \leq \lim_{K\to\infty}\mathbb E\bigl(\|(r(T\wedge \sigma_K),v(T\wedge \sigma_K)\bigr)\|^2)\leq C_T'
\end{align*}
can be concluded by means of the monotone convergence theorem. Note that $\mathbb E(\gamma_T)<\infty$ implies $\mathbb P(\gamma_T=\infty)=0$ and therefore $\mathbb P(\sigma\leq T)=0$ for all $T\geq 1$. Hence, we have shown that $\mathbb P(\sigma=\infty)=1$. Thus, $(r(t),v(t))_{t\in [0,\infty)}$ is the global solution of the SDE \eqref{eq:explicit_SDE} with $(r(t))_{t\in [0,\infty)}$ taking values in $\mathcal M$ and $(v(t))_{t\in [0,\infty)}$ taking values in the tangent space of the manifold. The uniqueness of the global solution follows from the uniqueness of the local solution. Since estimate \eqref{ineq:estimatesupnorm} is independent of $K$, we can show the bound \eqref{ineq:sup} for the global solution by help of monotone convergence,  
\begin{align*}
 \mathbb E\sup_{t\leq T}\|(r(t),v(t))\|^4&=\mathbb E\sup_{t\leq T}\|(r(t\wedge \sigma),v(t\wedge \sigma))\|^4=\lim_{K\to\infty}\mathbb E\sup_{t\leq T}\|(r(t\wedge \sigma_K),v(t\wedge \sigma_K))\|^4\\
&\leq C\bigl(\mathbb E\|(r(0),v(0))\|^4+ C_T^2 T^3\bigr)\operatorname{exp}\bigl(CC_T^2 T^3\bigr).
\end{align*}
\end{proof}

\begin{remark}\label{rem:onesided_growth}
The crucial observation \eqref{eq:onesidedlineargrowth} in the proof of Proposition~\ref{globalsolution} implies that the drift term in the explicit SDE~\eqref{eq:explicit_SDE} satisfies a one-sided linear growth condition on the constraint manifold in the sense that
\begin{align}\label{eq:onesidedlineargrowth2}
\begin{split}
\Bigl\langle \bigl(x,y\bigr),\bigl(y,P(x)a(t,x,y)+\nabla g(x)G^{-1}(x)D^2g(x)(y,y)\bigr)\Bigr\rangle\leq C_T \Bigl(1+\bigl\|(x,y)\bigr\|^2\Bigr)
\end{split}
\end{align}
whenever $x\in \mathcal M$, $y\in T_x\mathcal M$ and $t\in [0,T]$, where $C_T$ is the constant appearing in \eqref{eq:lineargrowth}. The strategy of our proof can be considered as a modification of the strategy for showing existence and uniqueness of global solutions to SDEs satisfying a \emph{global} one-sided linear growth assumption on the drift coefficient, see, e.g., \cite[Theorem 3.5]{mao:b:2007}.
\end{remark}

%%%%%%%%%%%%%%%%%%%%%%%%%%%%%%
%%%
\section{Numerical Investigations}
\label{sec:numerics}
\setcounter{equation}{0} \setcounter{figure}{0}

The theoretical analysis of our manifold-valued SDE is complemented by numerical studies concerning the time discretization. 
Strong explicit and implicit schemes, their convergence rates and stability behavior are well-studied for stochastic differential equations whose drift and diffusion coefficients satisfy (local, one-sided) Lipschitz and growth conditions, see, e.g., \cite{hutzenthaler:p:2015, kloeden:b:1992, kloeden:b:1994} and references within. However, for constrained SDE of the type \eqref{eq:intro_SDE} we lack analytical results reported in literature. In the following we investigate the strong convergence order of an implicit Euler scheme and of an explicit projection-based method and discuss their applicability to the turbulence-driven fiber dynamics regarding accuracy and computational effort. Moreover, we comment on the numerical treatment of the underlying constrained stochastic partial differential equation.

%%%
\subsection{Time-integration schemes}

We consider a finite time interval $[0,T]$, partitioned into subintervals $[t^n, t^{n+1}]$, $n = 0,...,M-1$, of fixed length $\Delta t=T/M$. The numerical approximation to the solution at a time level $t^n$ is indicated by the respective index $^n$, e.g., $r^n\approx r(t^n)$. As possible time-integration schemes, we use two different approaches: (A) an implicit Euler scheme based on the original formulation of the manifold-valued SDE \eqref{eq:intro_semidisc}, (B) an explicit projection-based Euler scheme for the formulation with the explicit representation of the Lagrange multiplier \eqref{eq:explicit_SDE}. 

In (A) we treat $(\mu^n-\mu^{n-1})/\Delta t=\lambda^n$  as Lagrange multiplier to the constraint $g(r^n)=0$. Evaluating all terms in \eqref{eq:intro_semidisc} implicitly --except of the diffusion coefficient $B$--  leads to a nonlinear system of equations at every time level which is solved via a Newton iteration equipped with an Armijo-Goldstein line-search. The implicit scheme enforces $r^n\in \mathcal{M}$ at every time level $t^n$. To avoid numerical errors and ensure $v^n\in T_{r^n}\mathcal{M}$ within the computational accuracy, we additionally apply the strategy by Gear-Gupta-Leimkuhler \cite{gear:p:1985} that is well-known for differential algebraic equations: the hidden constraint $\nabla g(r)^\top v=0$ is included together with introducing a second Lagrange multiplier $\nu$ in the dynamic equation for $r$, i.e., $\mathrm{d}r= (v+ \nabla g(r) \nu) \,\mathrm{d} t$.

In (B) we proceed from $(r^n,v^n)\in \mathcal{M}\times T_{r^n}\mathcal{M}$. Integrating  \eqref{eq:explicit_SDE} explicitly yields a new position and velocity $(\hat r, \hat v)$ which neither fulfill the algebraic nor the hidden constraint exactly. To enforce both within computational accuracy, we first project $\hat{r}$ on $\mathcal{M}$ by solving 
$r^{n+1}= \hat{r} + \nabla g(r^{n+1}) \eta$, $g(r^{n+1}) = 0$
via a Newton iteration. Note that the occurring Lagrange multiplier $\eta$ has no relation to $\lambda$.  Then we project $\hat{v}$ orthogonally onto the tangent space, i.e., $v^{n+1}= P(r^{n+1}) \hat{v}$, $v^{n+1}\in T_{r^{n+1}}\mathcal{M}$. The explicit time integration is computationally cheaper than the implicit one, but stability issues require a step size restriction. The scheme is stabilized without further costs by means of a semi-implicit modification where the dynamical equations for $r$ and $v$ are sequentially solved and $\hat{v}$ is determined using $r^{n+1}$ in the drift terms. 

Aiming for speed-up and reduction of the computational costs we investigate two further variants of the presented schemes.  The performance of the implicit scheme (A) depends strongly on the initial guess in the Newton iteration for which usually the solution of the previous time level and $\lambda=0$ are taken. By using instead an explicit Euler step as estimator we propose a predictor-corrector strategy in the spirit of a WIGGLE-algorithm \cite{lee:p:2005} known from molecular dynamic simulations. Additionally we explore the accuracy of the explicit scheme (B) when the projections onto $\mathcal{M}\times T_r\mathcal{M}$ are only applied after a number of time steps. 

%%%
\subsection{Performance study -- convergence and costs}

\begin{figure}[t]
 \centering
\includegraphics[scale=0.45]{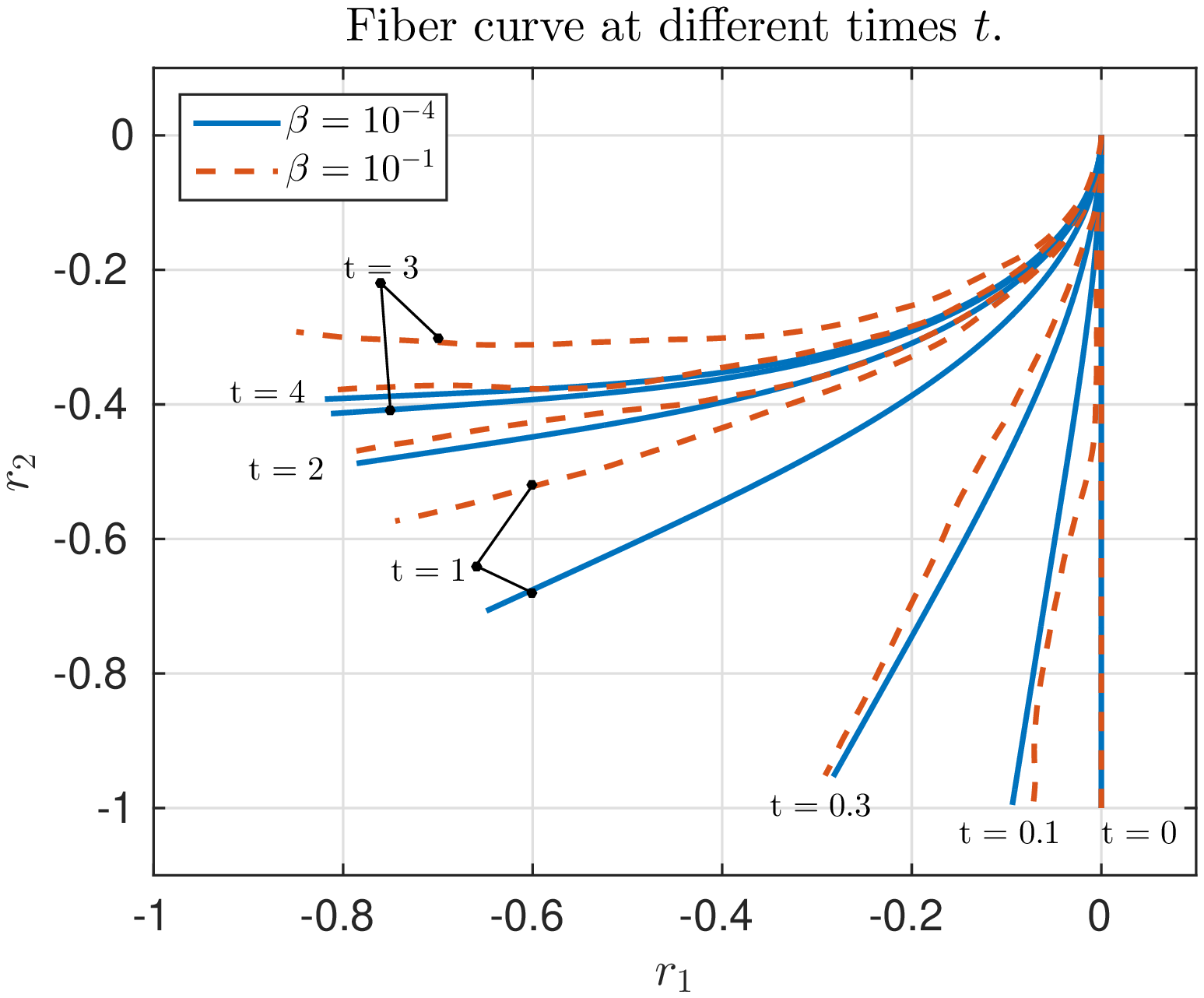}~\includegraphics[scale=0.45]{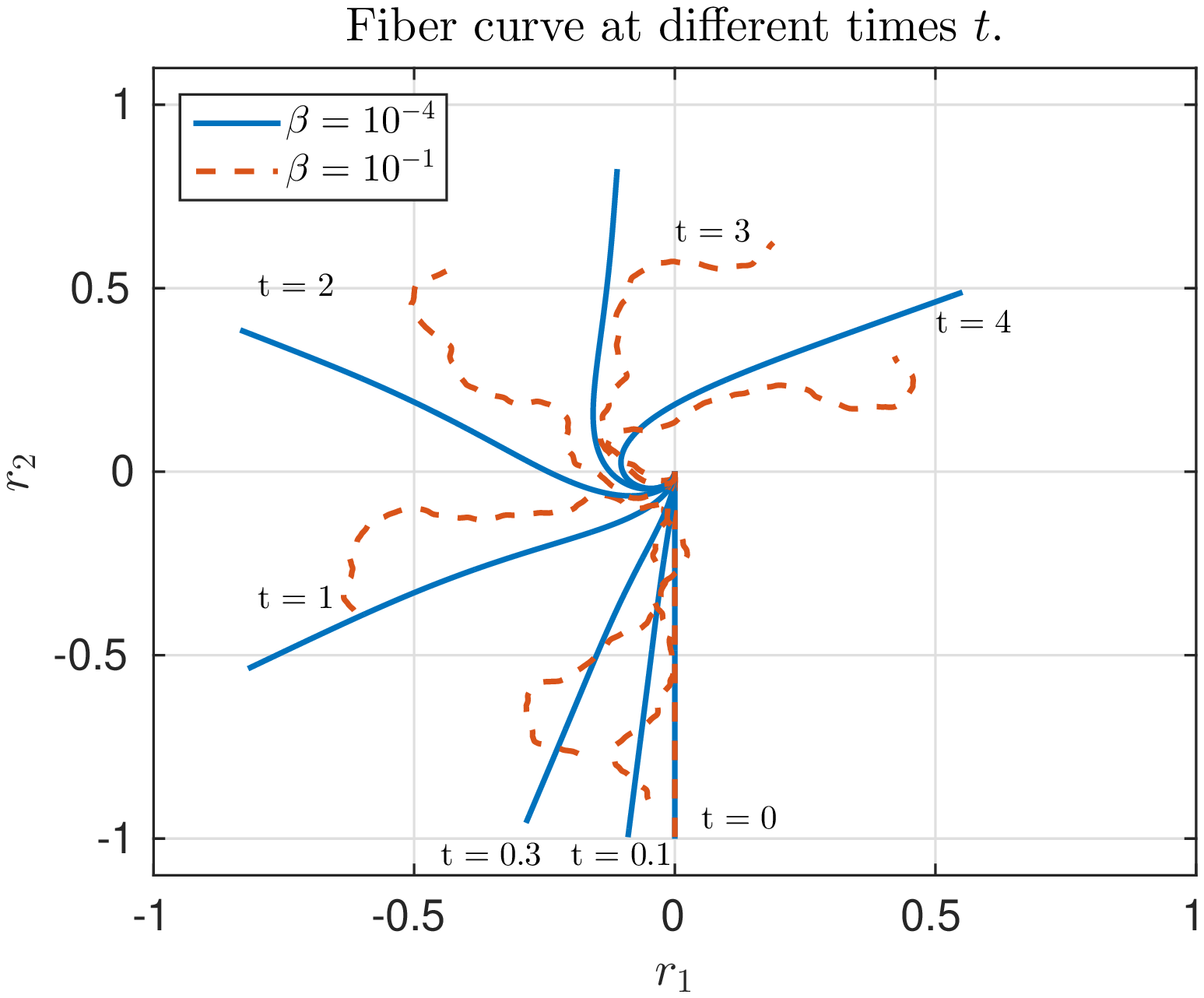}\\
\includegraphics[scale=0.45]{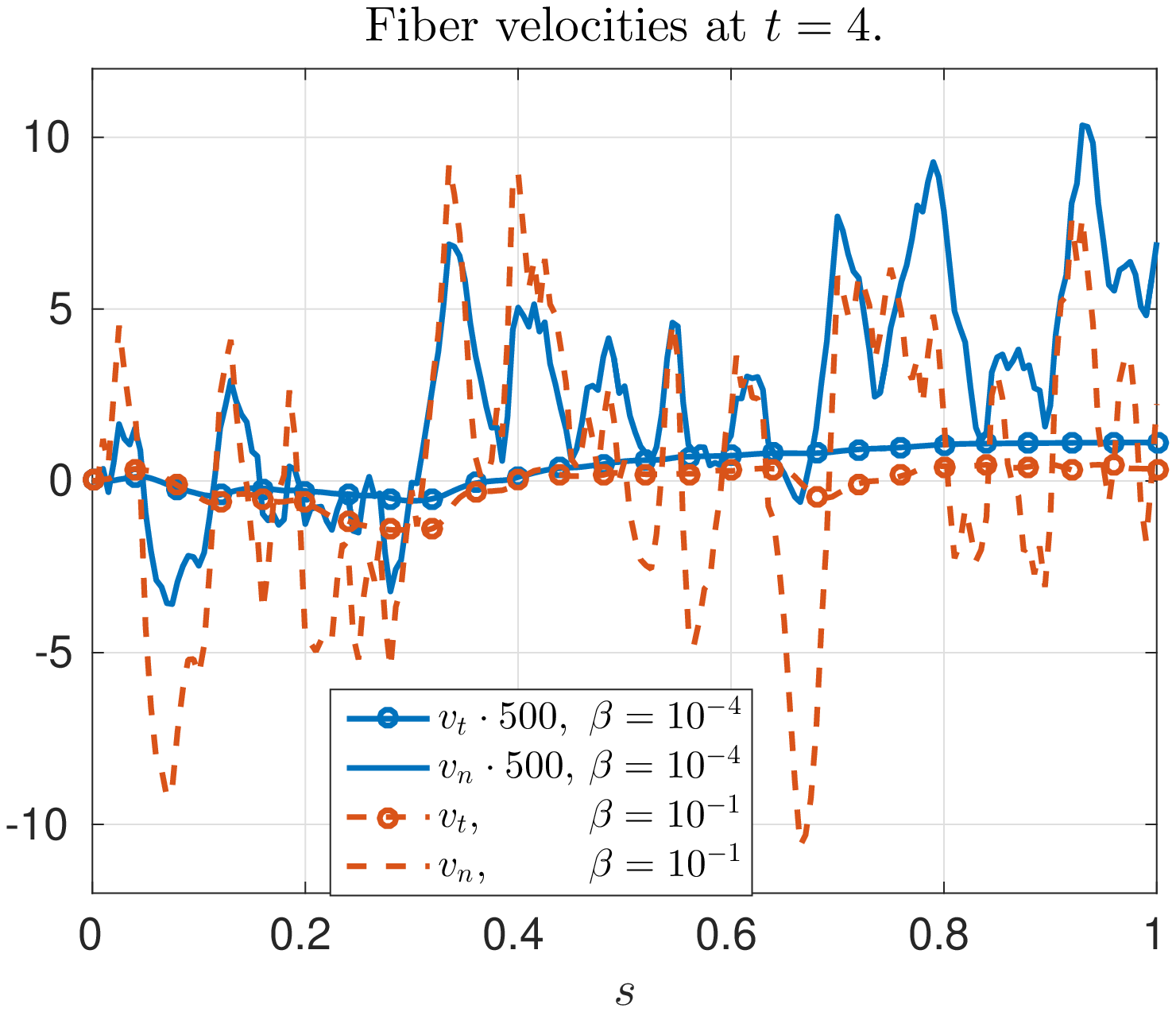}~\includegraphics[scale=0.45]{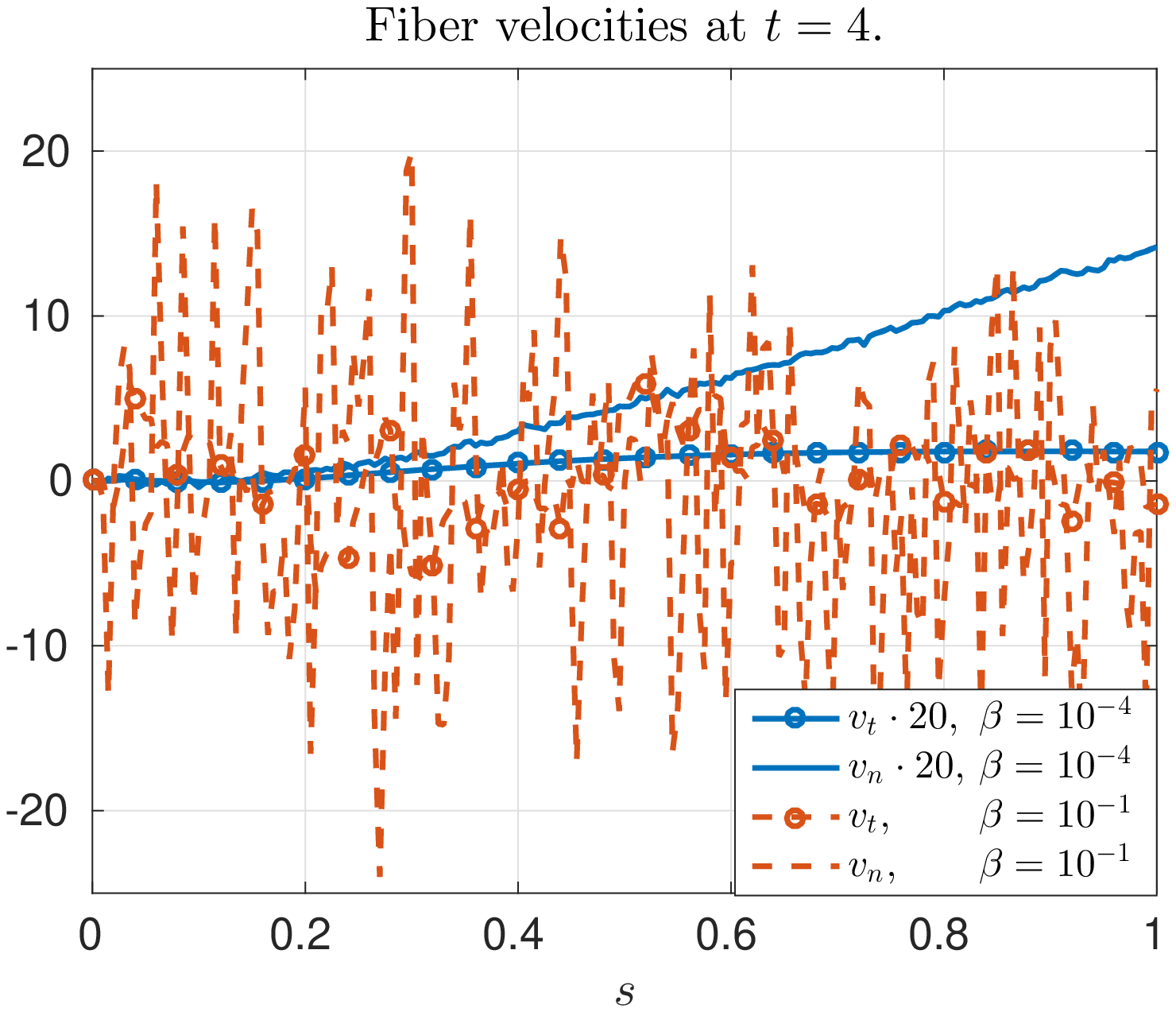}
\caption{\small Fiber curve and velocity for $\alpha=0.4$ ({\em left}) and $\alpha=4$ ({\em right}), computed with $(\Delta s, \Delta t) = (5\cdot 10^{-3}, 2.5\cdot10^{-7})$ using scheme (A). The velocity is split into the tangential and normal parts, i.e., $v_t=\mathbf{v}^\top\partial_s\mathbf{r}$ and $v_n=\mathbf{v}^\top\mathbf{n}$ with $\mathbf{n}\perp \partial_s\mathbf{r}$. Note that the velocity associated with $\beta=10^{-4}$ is scaled by a constant factor $c=500$ ({\em left}), $c=20$ ({\em right}).}
\label{fig:dynamics}
\end{figure}

For the performance study of the numerical schemes we consider the fiber dynamics in a two-dimensional set-up ($\mathbf{e_1}$-$\mathbf{e_2}$-plane with gravitational direction $\mathbf{e_g}=-\mathbf{e_2}$). The fiber is clamped at one end in the origin and hangs initially straight in direction of gravity, i.e., $\mathbf{r}(s,0)=-s\mathbf{e_2}$, $s\in[0,\ell=1]$, before its free end is excited into motion by a stationary rotational flow field $\mathbf{u}(\mathbf{x},t)\equiv\mathbf{u}(\mathbf{x})=x_2\mathbf{e_1}-x_1\mathbf{e_2}$. For simplicity, the drag operators $\mathbf{C}$, $\mathbf{D}$ are set to be identity. Moreover, the Froude and drag numbers are exemplarily chosen as $(\mathrm{Fr},\mathrm{Dr}) =(3, 0.1)$. 

The effects of bending stiffness and turbulent intensity on the fiber motion are illustrated in Fig.~\ref{fig:dynamics}, i.e., $\alpha\in\{0.4,4\}$, $\beta\in\{10^{-4},10^{-1}\}$. As expected, a higher turbulent fluctuation number $\beta$ implies higher fluctuations in the velocity and a distinct crimp of the fiber curve. The smaller $\alpha$, the more pronounced is the bending stiffness. In the following we fix $\alpha=0.4$.

\begin{figure}
\centering
\includegraphics[scale=0.45]{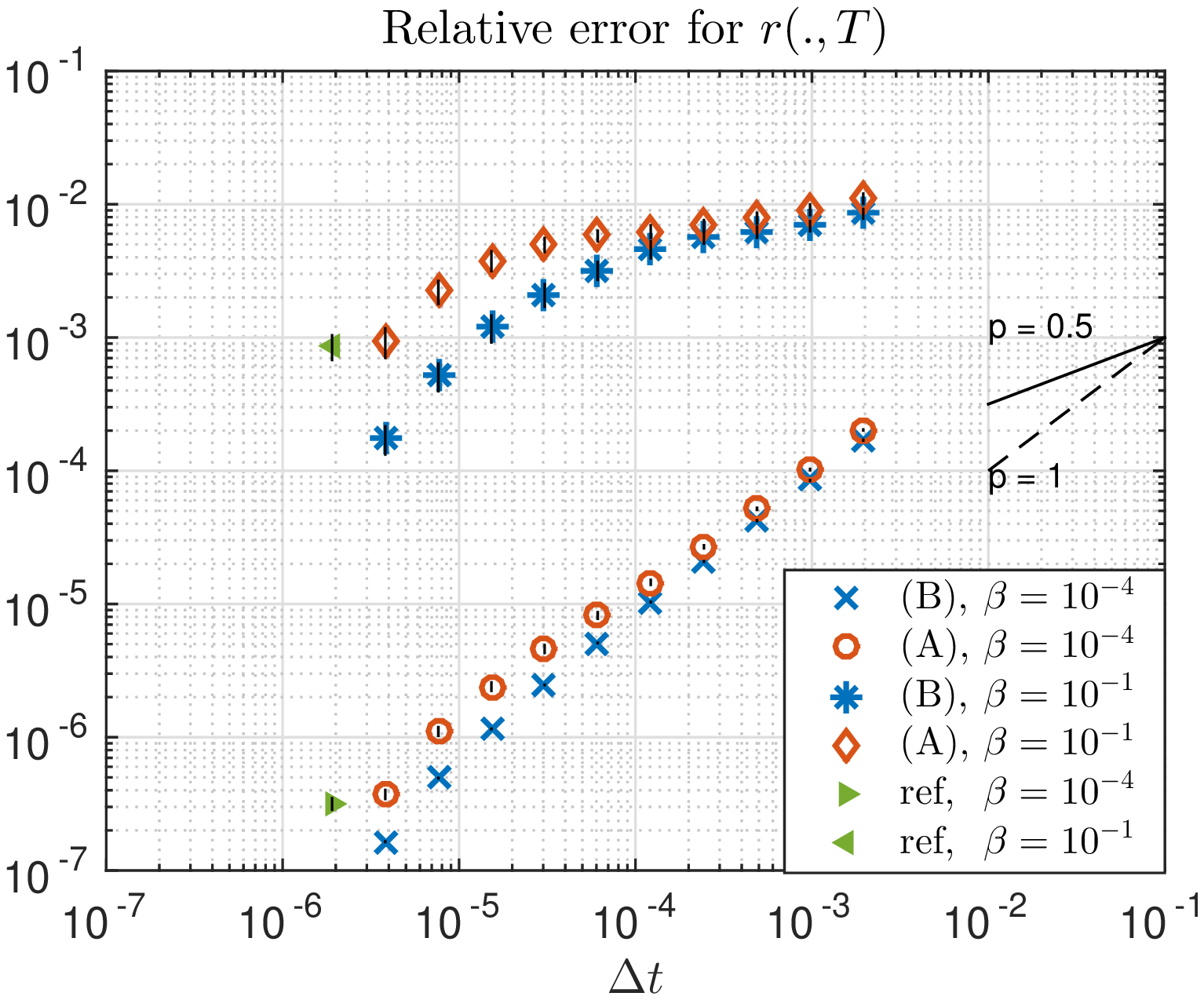}~\includegraphics[scale=0.45]{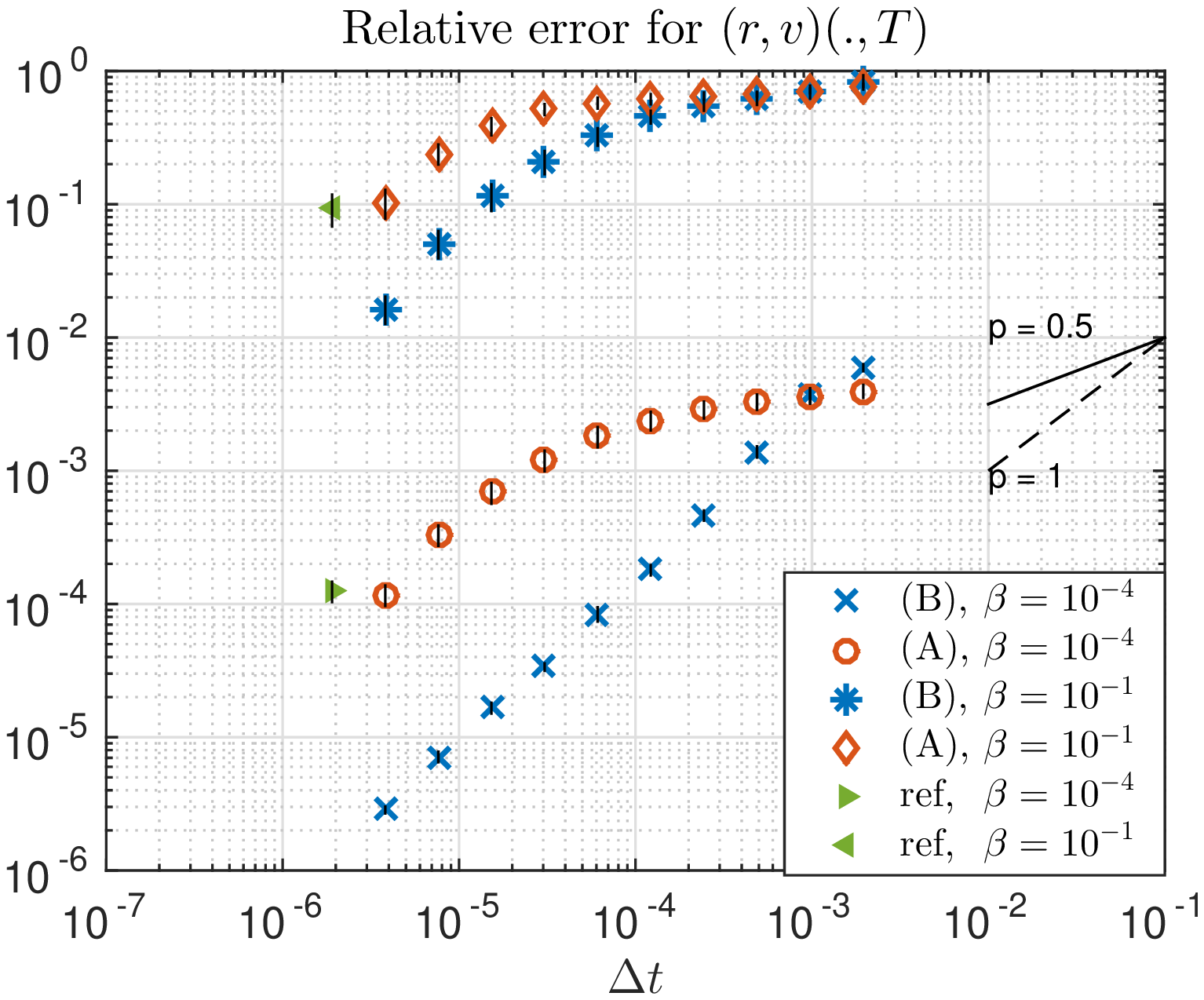}
\caption{\small Convergence of polygonal chain for $\Delta t \rightarrow 0$, fixed $\Delta s = 0.125$, $T=0.25$. The mean relative errors are shown for implicit (A) and explicit (B) schemes. The distance between the two corresponding reference solutions is marked by \textit{ref}.}
 \label{fig:polychainConv}
\end{figure}

\begin{figure}
\centering
\includegraphics[scale=0.45]{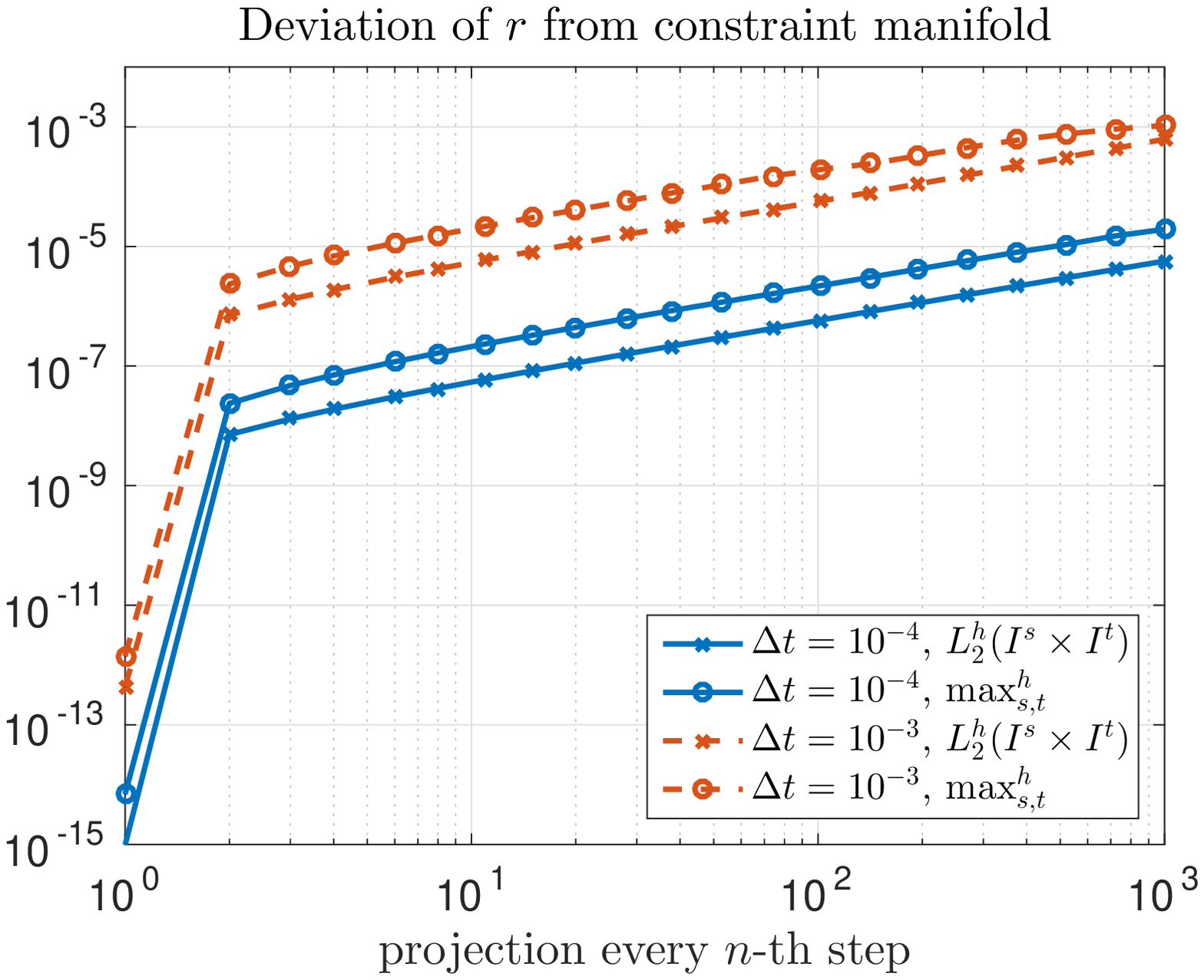}~\includegraphics[scale=0.45]{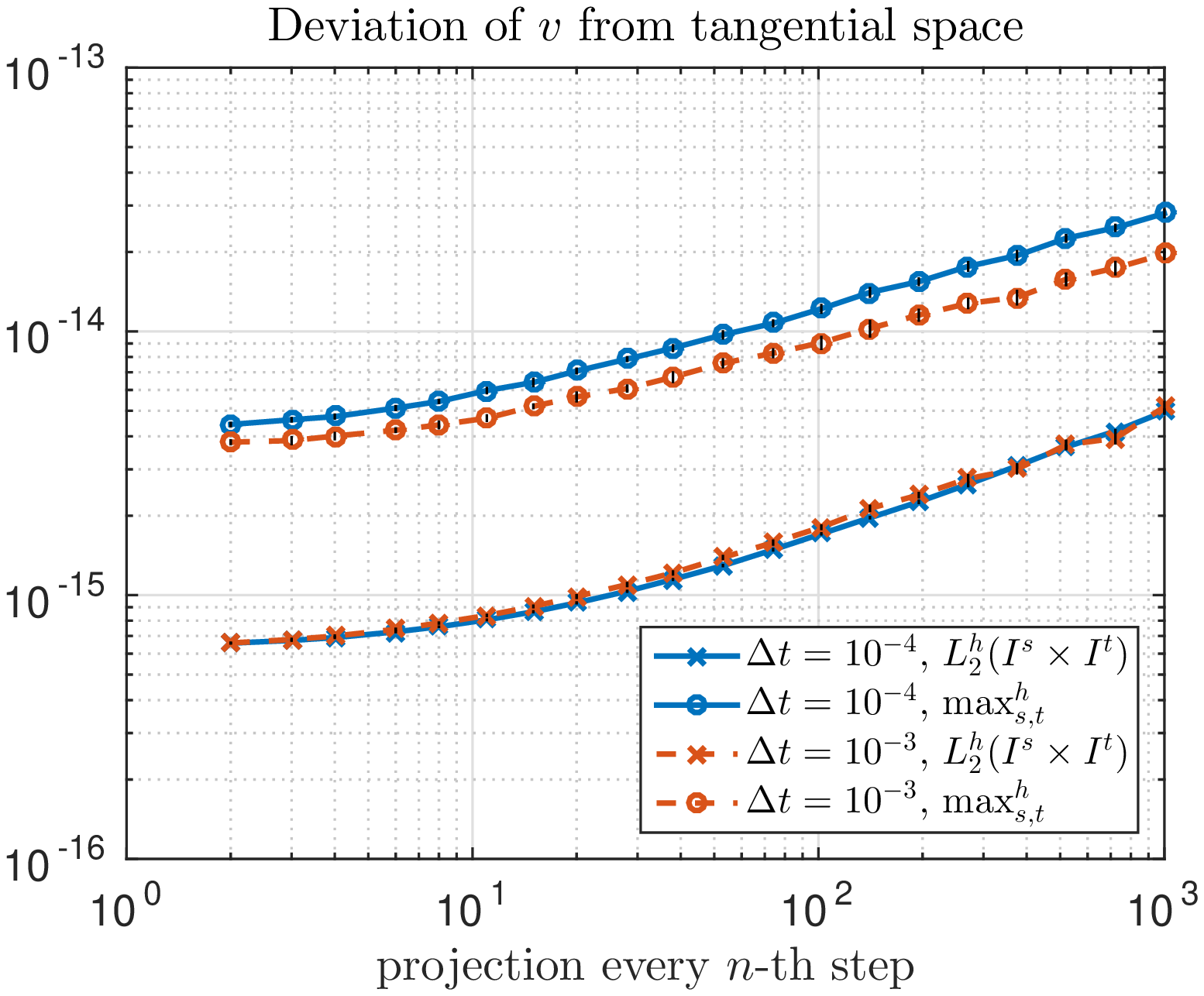}\\
\includegraphics[scale=0.45]{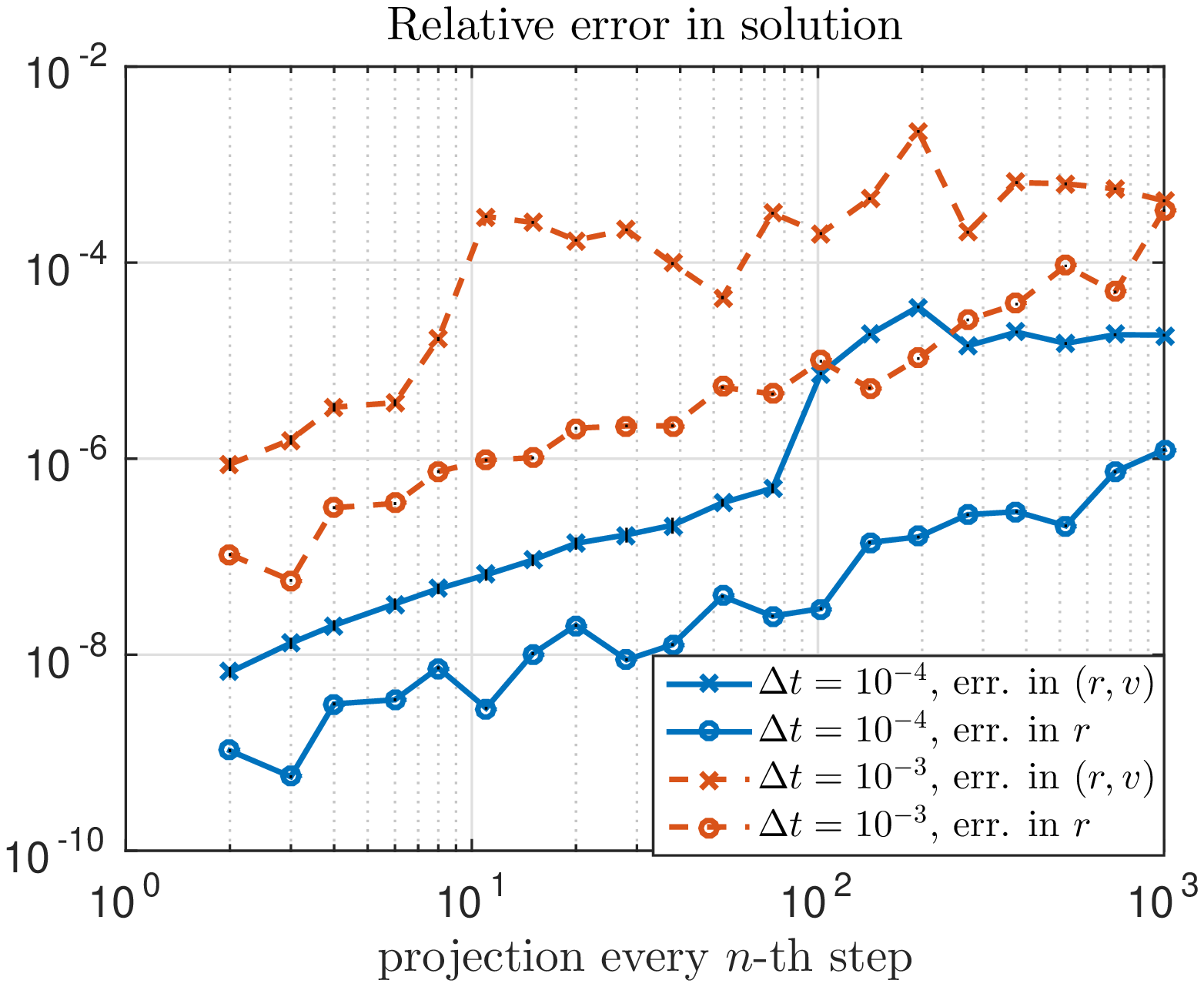}~\includegraphics[scale=0.45]{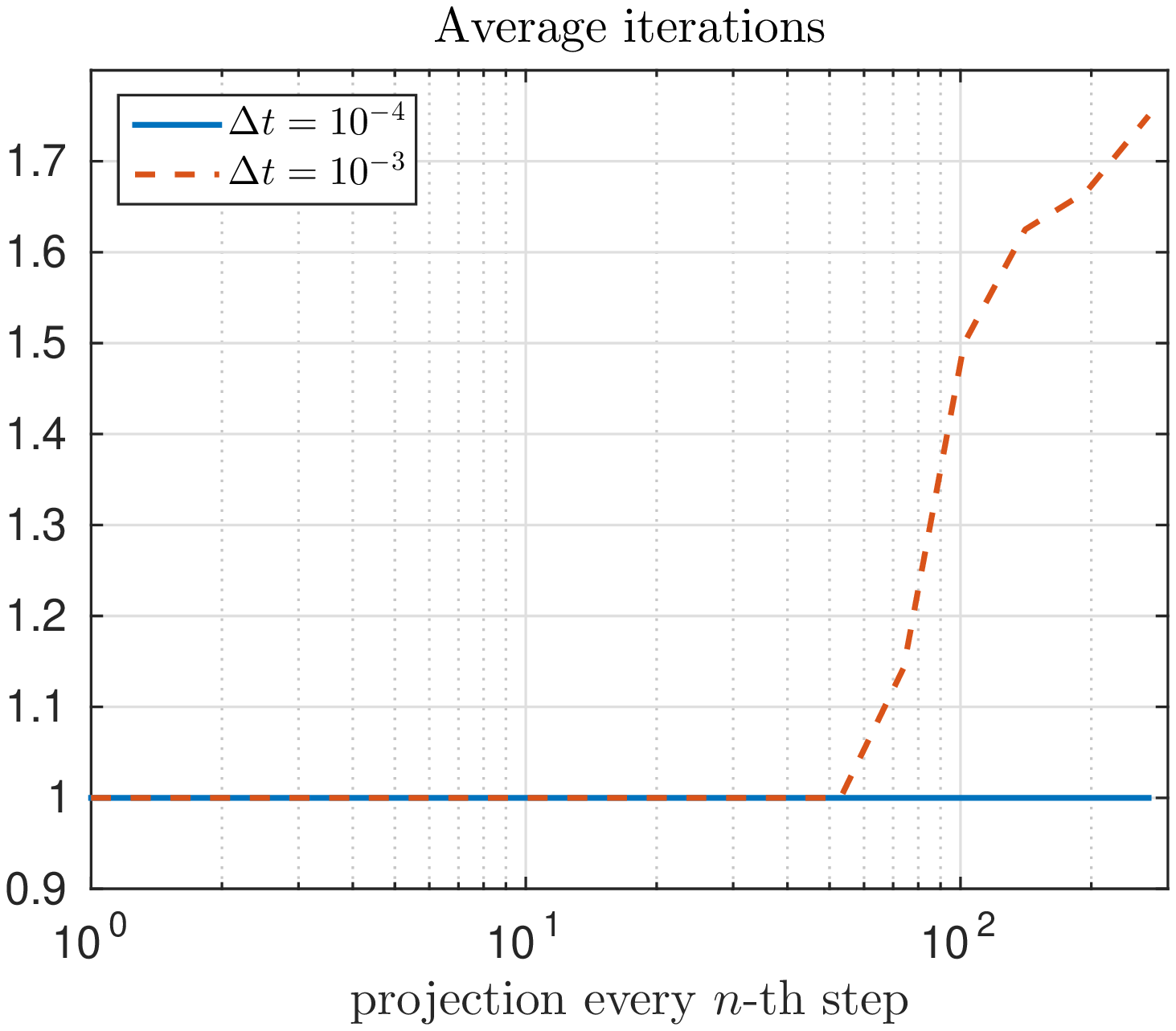}
 \caption{\small Influence of projection-skipping for $n$ steps in the explicit scheme. {\em Top}: mean error in the length constraint $\|\|\partial_s\mathbf{r}\|^2-1\|_{L_2^h(I^s\times I^t)}$ ({\em left}); mean deviation of velocity from $T_r\mathcal M$, $\|\partial_s\mathbf{r}^\top\partial_s\mathbf{v}\|_{L_2^h(I^s\times I^t)}$ ({\em right}). {\em Bottom}: mean error in the solution at $t=T$ ({\em left}), average number of Newton iterations for the $r$-projection ({\em right}). $\Delta s = 0.125$, $T=1$ and $L_2^h(I^s\times I^t)$ discrete approximation of the $L_2$-norm on $[0,\ell]\times[0,T]$.}
 \label{fig:skippingError}
\end{figure}

To study the strong numerical convergence of the two proposed integration schemes, we analyze the relative error $\|(z-z_{ref})(T)\|_{L_2^h(I^s)}/\|z_{ref}(T)\|_{L_2^h(I^s)}$ for $z\in\{r,(r,v)\}$, where the discrete approximation of the $L_2$-norm reads as $\|z\|_{L_2^h(I^s)}=(\sum_{i=1}^N\|\mathbf z_i\|^2\Delta s)^{1/2}$ with Euclidean norm $\|\cdot\|$. Due to the lack of an analytical solution, we calculate the reference solution, denoted by index $._{ref}$, by help of the respective scheme under consideration. The mean behavior is obtained by averaging a sample of 100 independent simulations, the confidence level of the presented results is set to $90\%$. Figure~\ref{fig:polychainConv} indicates a strong convergence rate of order $p\approx1$ for the explicit projection-based and the implicit schemes in both, position and velocity for turbulent numbers $\beta\in\{10^{-4},10^{-1}\}$. Furthermore, we note that an increase of $\beta$ results in a raise in the magnitude of the mean relative error for both schemes (A) and (B). The distance between the two reference solutions shows good agreement of the schemes, see Fig.~\ref{fig:polychainConv}. The implicit scheme requires the solving of a nonlinear system with a Jacobian of size $((2d+2)(N+4))^2$ in every time step, where $d$ denotes the space dimension $d\in\{2,3\}$. In comparison, only a tridiagonal linear system of size $N^2$ is solved in an explicit step. The additional projections for $r$ onto the constraint manifold $\mathcal M$  and for $v$ onto the tangential space $T_r\mathcal{M}$ involve the solving of a nonlinear system with a Jacobian of size $((d+1)(N+4))^2$ and of a linear system with Gram matrix $G$ of size $N^2$, respectively. The corresponding computational time is visualized as fraction of the time needed for a Newton step in the implicit scheme, see Fig.~\ref{fig:avgNewtonEstimator} (right). In total the explicit projection-based scheme shows a computational effort up to $\approx70\%$ less than the implicit scheme in every time step, but is subjected to a step size restriction due to stability issues.

In the stability region we can even speed up the explicit scheme (B) by skipping the projections of $r$ and $v$ for $n$ time steps, $n\in \mathbb{N}$. Certainly this leads in general to an increase of the numerical errors in the constraints and also in the computed solution. Additionally the intended speed up might be repealed due to a dramatical increase of the needed Newton iterations in the projection step of $r$ onto $\mathcal M$. However, as Fig.~\ref{fig:skippingError} indicates, the mean errors grow relatively slow with $n$, and the average amount of needed Newton iterations is less than 2 even if the projections are suspended for over 100 time steps.

\begin{figure}[t]
\centering
\includegraphics[scale=0.45]{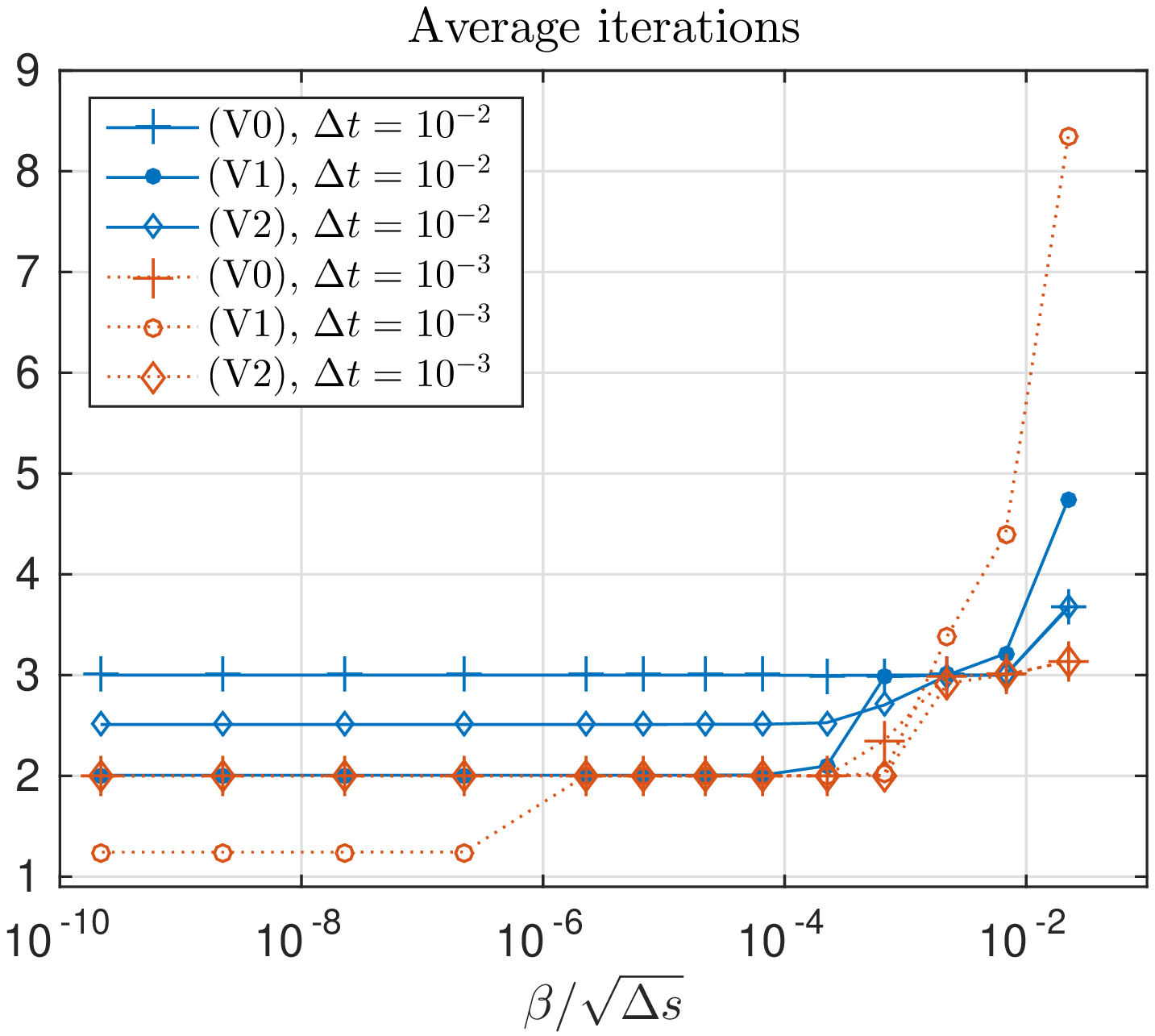}~\includegraphics[scale=0.45]{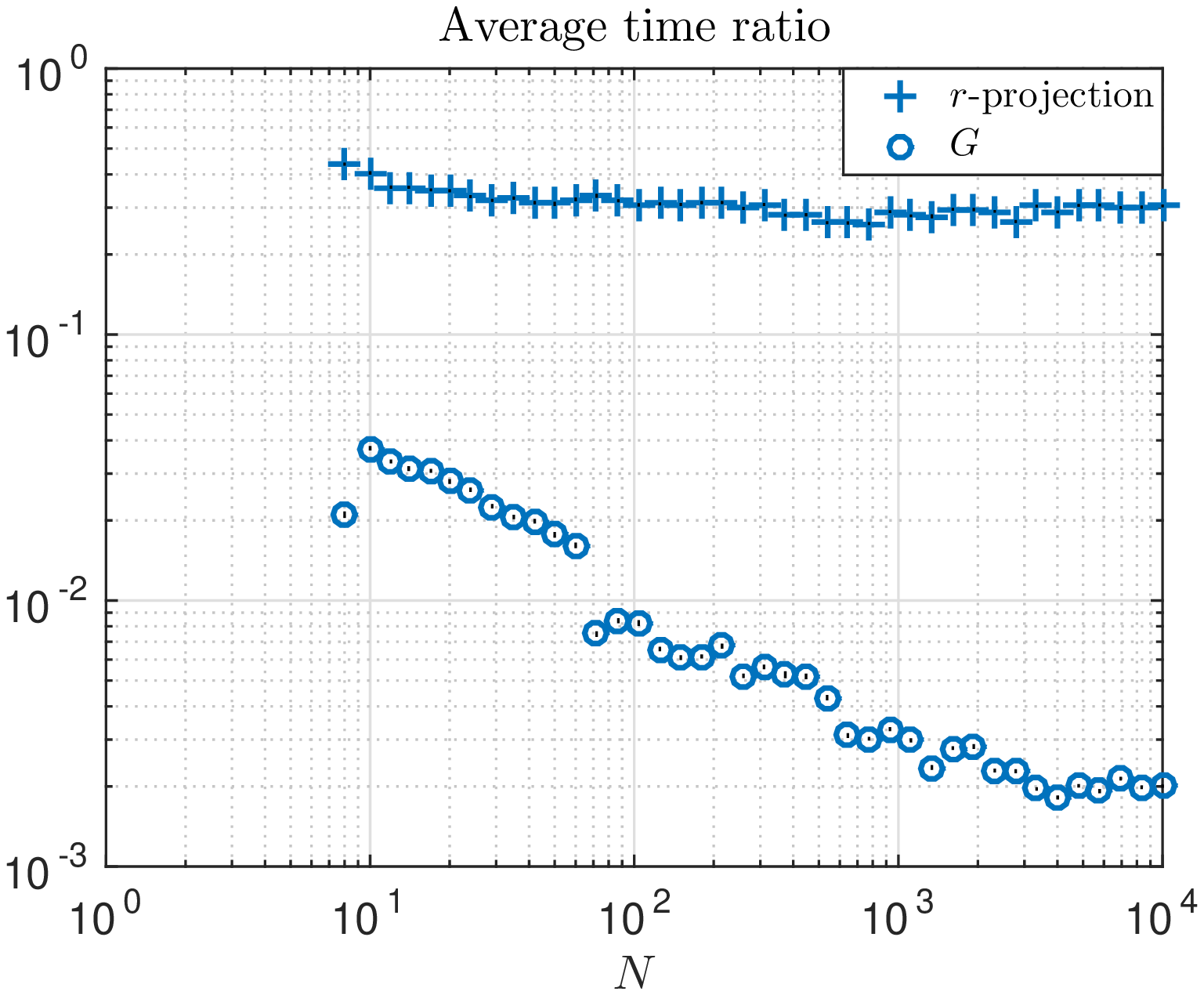}
 \caption{\small {\em Left:} Average of Newton iterations w.r.t.\ effective turbulence number for different initial guesses of the implicit solution. $T=3$, $N = 20$. {\em Right}: Average computation time of the linear systems (associated with the Gram matrix $G$ and the $r$-projection) as fraction of time needed for a Newton step in scheme (A) in dependence of the number of spatial grid points $N$.}
 \label{fig:avgNewtonEstimator}
\end{figure}

The implicit scheme (A) is applicable for larger step sizes. Its computational performance can be improved by using the explicit Euler step as initial guess for the Newton method. We particularly forego the projections to avoid unnecessary computations. We explore two different variants (V1) and (V2) of this approach. In (V1) we estimate the complete tuple $(r,v,\lambda)$, in (V2) only the Lagrangian multiplier $\lambda$ whereas position and velocity are taken from the old time level. The impact of the initial guess on the Newton iterations needed in average depends on the effective turbulence number $\beta/\sqrt{\Delta s}$, as it can be seen in Fig.~\ref{fig:avgNewtonEstimator} (left). For small stochastic influence, (V1) results in a reduction of Newton iterations ($\approx1$) which goes with less costs due to the significant size differences of the respective system matrices as already discussed. For high turbulence numbers or in the case $\Delta s\ll \Delta t$, we observe a disastrous estimation of the velocity. This leads to even more Newton iterations than in the standard case (V0) where the fixed point iteration is initialized with $r,v$ of the old time level and $\lambda=0$. The reason lies in the stability behavior of the explicit scheme. In comparison, (V2) is robust to the effective turbulence number, but not as efficient as (V1) for small stochastic influence.

%%%%%%%%%%%%%
%%%
\section{Conclusion}
\setcounter{figure}{0}

For the dynamics of a slender, elastic, inextensible fiber in turbulent flows we investigated a spatially discrete surrogate model. It was formally deduced from a continuous space-time Kirchhoff beam model by help of a finite volume approach, and rigorously interpreted as an It\^o-type SDE with a nonlinear algebraic constraint. We proved existence and uniqueness of a global solution and proposed two numerical schemes for the manifold-valued stochastic differential system in time. The proof of the existence and uniqueness result was based on an explicit representation of the Lagrange multiplier, a detailed analysis of the occurring explicit drift-coefficient, and a Gronwall-type argument, showing that the lifetime of the local solution is infinite with probability one. The analytical investigation motivated the introduction of an explicit projection-based Euler-type scheme which is subjected to a time step restriction, but --in the stability region-- more efficient than a respective implicit scheme. The performance of the implicit scheme which is generally applicable --also for larger time steps-- can be improved for small stochastic effects by means of a predictor-corrector strategy where an explicit Euler step is used as initial guess.  Both schemes show numerically a strong convergence rate of order $p\approx1$ for the model problem. The numerical behavior of the spatially discrete surrogate model in the limit  $\Delta s=\Delta t\to0$ indicates a strong convergence of the continuous space-time model: a convergence rate of order $p=1$ can be observed, cf.\ Fig.~\ref{fig:coupledConv}. However, an analytical result for existence and uniqueness of solutions for $\Delta s\to0$, i.e., a solution theory for the constrained stochastic partial differential fiber model, is still a challenge for future research. 

\begin{figure}[t]
\centering
\includegraphics[scale=0.45]{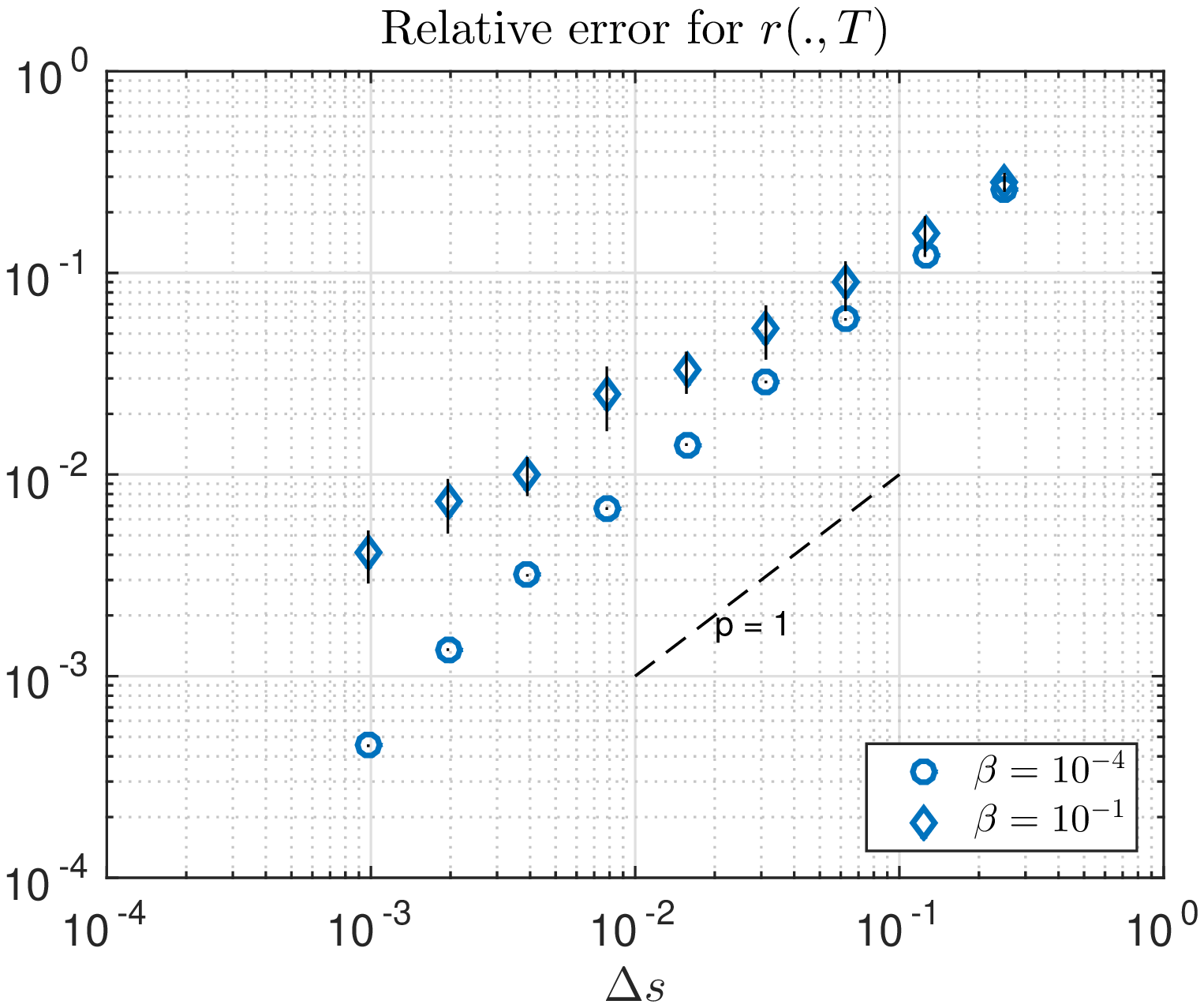}~\includegraphics[scale=0.45]{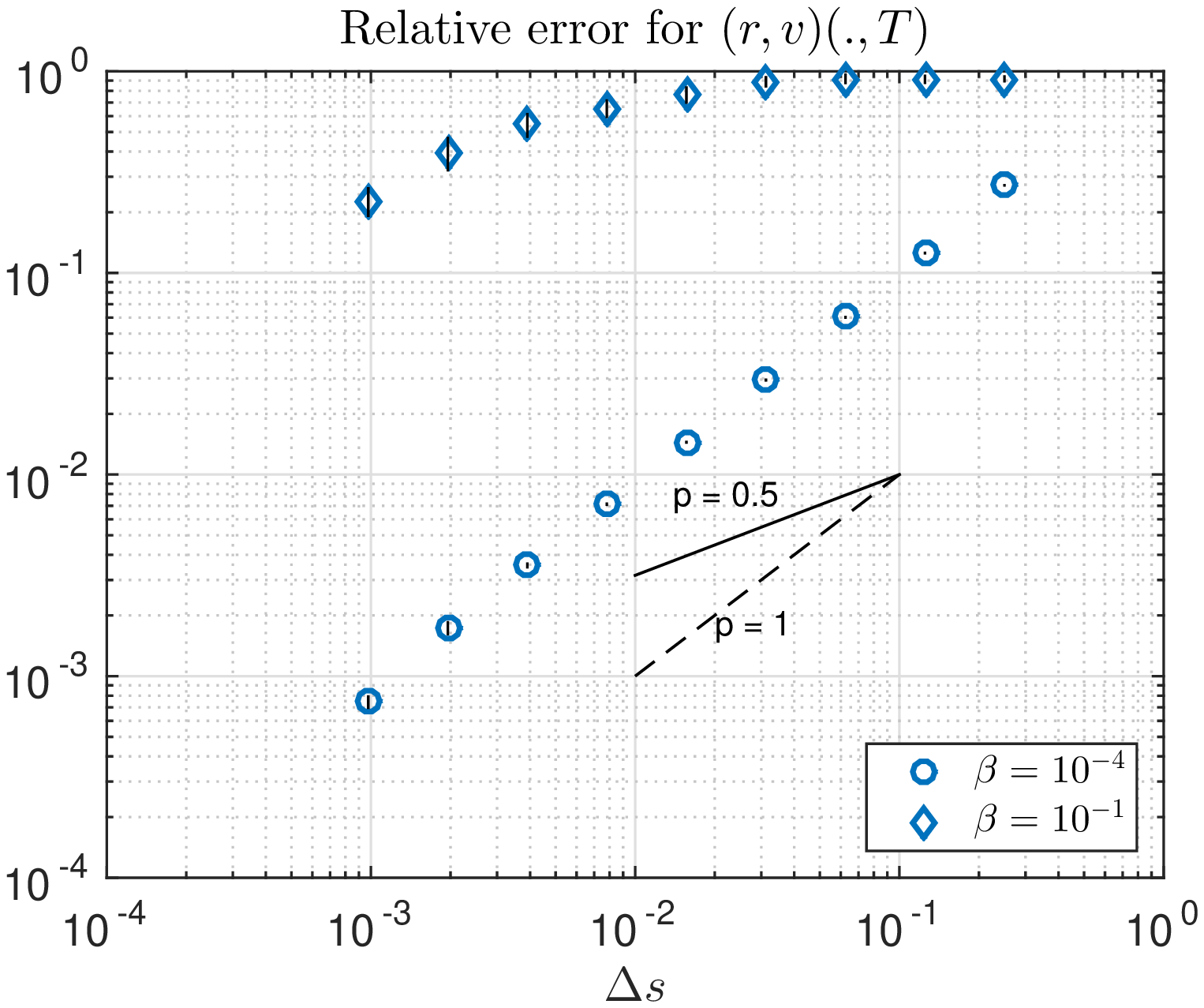}
\caption{\small Space-time convergence for $\Delta t=\Delta s\rightarrow 0$, computed with the finite volume approach in combination with the implicit time scheme (A), $T=2.5$, sample of 20 simulations.}
 \label{fig:coupledConv}
\end{figure}

\subsection*{Acknowledgement}
The support by the German Bundesministerium f\"ur Bildung und Forschung (BMBF) and the Deutsche Forschungsgemeinschaft (DFG) is acknowledged (Project OPAL 05M13, Project MA 4526/2-1, WE 2003/4-1).

%%%%%%%%%%%%%%%%%%%%%%%%%%%%%%
%\bibliographystyle{siam}
%\bibliography{ref_2015}
%%%%%%%%%%%%%%%%%%%%%%%%%%%%%

\end{document}